\documentclass[final,hidelinks,onefignum,onetabnum]{siamart251216}

\usepackage[T1]{fontenc}
\usepackage[utf8]{inputenc}
\usepackage{lmodern}
\usepackage{amssymb,mathtools}
\usepackage{microtype}
\usepackage{booktabs,multirow}
\usepackage{xurl}
\usepackage{placeins}

\newsiamremark{assumption}{Assumption}
\newsiamremark{remark}{Remark}
\crefname{assumption}{Assumption}{Assumptions}
\crefname{remark}{Remark}{Remarks}

\newcommand{\R}{\mathbb{R}}
\newcommand{\sphere}{\mathbb{S}}
\newcommand{\K}{\mathbb{K}}
\newcommand{\bs}[1]{\boldsymbol{#1}}
\newcommand{\X}{\bs{X}}
\newcommand{\Y}{\bs{Y}}
\newcommand{\n}{\bs{n}}
\newcommand{\Pmat}{\bs{P}}
\newcommand{\Amat}{\bs{A}}
\newcommand{\Z}{\bs{Z}}
\newcommand{\avec}{\bs{a}}
\newcommand{\vvec}{\bs{v}}
\newcommand{\wvec}{\bs{w}}
\newcommand{\om}{\bs{\omega}}
\newcommand{\q}{\bs{q}}

\newcommand{\N}{\bs{N}}
\newcommand{\id}{\bs{\mathrm{id}}}
\newcommand{\dd}{\,\mathrm{d}}
\newcommand{\grad}{\nabla_\Gamma}
\newcommand{\lap}{\Delta_\Gamma}
\newcommand{\dt}{\partial_t^{\bullet}}
\newcommand{\dX}{\delta\X^{m+1}}
\newcommand{\dH}{\delta H^{m+1}}
\newcommand{\B}{\mathcal{B}}
\newcommand{\D}{\mathcal{D}}
\DeclareMathOperator{\tr}{tr}

\DeclareMathOperator{\sym}{sym}

\headers{SP-PFEM for coupled geometric flows}{Y. Zhang}
\title{A unified structure-preserving framework for geometric flows with coupled orientation and curvature dependence}
\author{Yulin Zhang\thanks{School of Mathematical Sciences and Institute of Natural Sciences,
Shanghai Jiao Tong University, Shanghai 200240, China
(\email{yulin.zhang@sjtu.edu.cn}).}}
\hypersetup{
  pdftitle={A unified structure-preserving framework for geometric flows with coupled orientation and curvature dependence},
  pdfauthor={Yulin Zhang}
}

\begin{document}
\maketitle

\begin{abstract}
We develop a structure-preserving parametric finite element framework
for geometric flows whose energy density couples the unit normal and
the curvature. This class includes bending energies with
orientation-dependent rigidity or spontaneous curvature. A common
fully discrete formulation treats closed curves in two dimensions and
closed surfaces in three dimensions, and accommodates the $L^2$ flow,
curve or surface diffusion, and the area- or volume-constrained $L^2$
flow. The formulation couples the geometric and curvature updates so
that their contributions satisfy a discrete energy inequality. It uses
continuous piecewise linear elements, a surface energy matrix, and a
mass-lumped projection of the curvature derivative of the density.
For positive densities satisfying a directional condition and convexity
in curvature, we prove energy dissipation without a time-step
restriction. The diffusion and constrained flows also preserve the
enclosed area or volume exactly. The analysis allows non-even
anisotropies and nonseparable dependence on orientation and curvature.
Numerical experiments exhibit approximately second-order convergence
in the manifold distance and confirm the
discrete structural properties. Shape relaxation under an anisotropic
Helfrich-type energy illustrates the use of the framework for coupled
directional and bending effects.
\end{abstract}

\begin{keywords}
Geometric flows, parametric finite element method, coupled orientation and
curvature dependence, energy dissipation, area and volume conservation
\end{keywords}

\begin{MSCcodes}
65M60, 65M12, 35K55, 53C44
\end{MSCcodes}

\section{Introduction}\label{sec:introduction}

The evolution of an interface is often governed by a competition between
surface tension and bending. Surface tension favors a reduction of
interfacial energy, while bending penalizes curvature or promotes a
preferred curvature. Directional dependence adds a further influence on
the evolving shape: the energetic cost of an interface can vary with its
orientation. These mechanisms appear in models of anisotropic interfaces
and elastic membranes. In particular, the classical bending energies of
Canham and Helfrich relate membrane shape to curvature
\cite{canham1970minimum,helfrich1973elastic}, while anisotropic surface
energies describe the preference for particular interface orientations.
Curvature dependence also provides a regularization of strongly
anisotropic interface evolution, as in the theory of Gurtin and Jabbour
\cite{gurtin2002interface}. Computing their combined effects calls for
numerical methods that retain
the dissipation and conservation laws governing the evolution.

In this work, we consider energies of the form
\begin{equation}\label{eq:intro-energy}
  W(\Gamma)=\int_\Gamma f(\n,H)\dd A,
  \qquad \Gamma\subset\R^d,\quad d=2,3,
\end{equation}
where $\n$ is the unit normal and $H$ denotes the signed curvature of a
planar curve or the sum of the principal curvatures of a surface. The
measure $\dd A$ denotes arclength for $d=2$ and surface area for $d=3$.
The density $f$ need not be a sum of an orientation-dependent term and a
curvature-dependent term. For example,
\begin{equation}\label{eq:intro-coupled-density}
  f(\n,H)=\gamma(\n)
       +\frac{\epsilon^2}{2}a(\n)\bigl(H-H_\ast(\n)\bigr)^2
\end{equation}
with $\gamma>0$ and $a>0$ allows the bending rigidity or preferred
curvature to depend on the
surface orientation. Such a choice is an anisotropic extension of the
classical Helfrich bending mechanism, with the anisotropy prescribed
relative to a fixed spatial frame. It provides a way to investigate how
directional preferences compete with bending during shape relaxation
(see Figure~\ref{fig:orientation-curvature}).

\begin{figure}[!htbp]
  \centering
  \includegraphics[width=\textwidth]{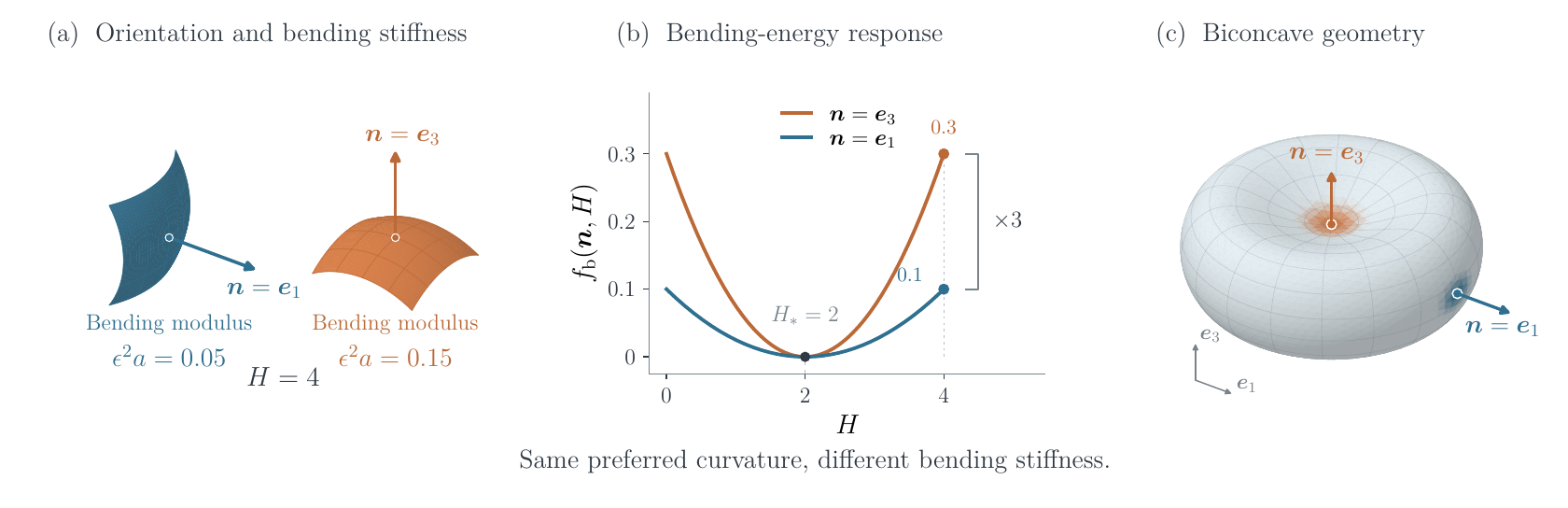}
  \caption{Orientation-dependent bending for
$f(\n,H)=1+\tfrac14\sum_{i=1}^3n_i^4
+\tfrac{\epsilon^2}{2}a(\n)(H-2)^2$,
with $a(\n)=1+2n_3^2$ and $\epsilon^2=0.05$.
(a) Patches with the same curvature $H=4$ and different central normals;
the bending moduli $\epsilon^2a(\n)$ are $0.05$ and $0.15$.
(b) The bending contribution $f_{\mathrm b}=f-\gamma$, with
$\gamma(\n)=1+\tfrac14\sum_{i=1}^3n_i^4$.
Both curves are minimized at $H=2$; at $H=4$, their values are $0.1$ and $0.3$.
(c) The corresponding normal directions on the biconcave initial geometry.
The colored neighborhoods identify the marked points, whose curvatures
need not agree.}
  \label{fig:orientation-curvature}
\end{figure}

In this paper, we study the $L^2$ and $H^{-1}$ gradient flows of
\eqref{eq:intro-energy}, together with the $L^2$ flow constrained to
preserve the enclosed area or volume. The $H^{-1}$ flow is referred to as
curve diffusion in two dimensions and surface diffusion in three
dimensions. All these flows dissipate the energy; diffusion and the
constrained $L^2$ flow also preserve the enclosed measure. These laws are
particularly useful in computations of shape relaxation, where an
accumulated volume error can alter the limiting shape and an artificial
increase in energy can obscure the balance between competing effects.
For a density with a nondegenerate curvature dependence, the $L^2$ and
diffusion flows are fourth- and sixth-order geometric equations,
respectively. Their high differential order and nonlinear coupling to
the evolving geometry make a simultaneous preservation of these laws
nontrivial after discretization.

Over the past few decades, various numerical methods have been developed
for curvature-driven interface evolution, including finite element
methods \cite{raetz2006surface}, finite difference methods
\cite{du2004phase}, level set methods
\cite{osher1988fronts,smereka2003semiimplicit}, phase field methods
\cite{du2004phase,raetz2006surface}, particle methods
\cite{leung2011particle}, finite volume methods
\cite{mikula2010flowing}, threshold dynamics \cite{merriman1994motion},
evolving surface finite element methods (ESFEMs)
\cite{kovacs2019convergent}, and parametric finite element methods
(PFEMs) \cite{barrett2008hypersurfaces,barrett2020parametric}. Among these
approaches, the energy-stable PFEMs developed by Barrett, Garcke, and
N\"urnberg (BGN) \cite{barrett2008hypersurfaces,barrett2020parametric}
are particularly attractive for their good mesh quality in computations
of mean curvature flow and surface diffusion. PFEMs combine direct
resolution of the interface with a variational treatment of its geometry.
Their unknowns lie on the evolving curve or surface, avoiding a bulk
grid and the resolution of a diffuse transition layer. The weak
formulation allows curvature and high-order terms to be treated with
low-order elements, while providing a systematic way to reproduce energy
dissipation and area or volume conservation. These features are
particularly useful when orientation and curvature jointly determine
the driving force.

Building on this variational framework, Bao and Zhao
\cite{bao2021structure} introduced a structure-preserving PFEM for
surface diffusion that combines unconditional energy stability with
exact conservation of the enclosed area or volume. Their use of a
time-averaged normal relates the discrete normal displacement to the
actual change in the enclosed measure.

For energies depending on orientation, surface energy matrices have
provided a systematic route to stable parametric approximations.
Bao and Li developed a symmetrized PFEM for anisotropic surface diffusion
in three dimensions \cite{bao2023symmetrized3D}, a structure-preserving
method for planar anisotropic geometric flows \cite{bao2024anisotropic},
and a unified treatment of anisotropic surface diffusion in two and three
dimensions \cite{bao2025unified}. Zhang, Li, and Ying
\cite{zhang2025stabilized} analyzed stabilized approximations for general
planar surface energies. Li, Ying, and Zhang \cite{li2025optimal}
established the optimal directional condition within the local energy
estimate framework, and Bao, Li, Ying, and Zhang \cite{bao2026alpha}
subsequently organized the planar surface energy matrices into an
$\alpha$-dependent family. These developments allow increasingly general
anisotropies to be treated while retaining discrete energy dissipation
and area or volume conservation.

Curvature-dependent energies have motivated a complementary line of
work. Dziuk \cite{dziuk2008willmore} developed a parametric finite element
method for Willmore flow with a stable spatial discretization.
Do\u{g}an and Nochetto \cite{dogan2012first} derived a general
first-variation formula for surface energies depending on geometric
quantities, including the normal and curvature. For Willmore flow, Bao
and Li \cite{bao2025planarWillmore} constructed an unconditionally
energy-stable PFEM in two dimensions. Bao, Li, and Wang
\cite{bao2025energy3D} developed an energy-stable method in three
dimensions using a weak evolution equation for the mean curvature, and
extended the approach to energies given by convex functions of the mean
curvature. Garcke, N\"urnberg, and Zhao \cite{garcke2026splitting}
constructed a fully discrete method for Willmore flow with spontaneous
curvature in two and three dimensions. Their normal--tangential velocity
splitting leads to uniquely solvable linear systems and an unconditional
energy stability estimate. Liu and Zhao
\cite{liu2026helfrich} recently proposed a structure-preserving PFEM
for the constrained Helfrich flow of curves and surfaces, with enclosed
volume and surface area constraints. Such formulations make it possible to work with piecewise linear finite elements despite the
higher order of the geometric evolution equations.

Extending these approaches to a nonseparable density
\eqref{eq:intro-energy} requires the directional and curvature responses
to be discretized together. A change in curvature alters the directional
energy, while a change in orientation alters the curvature response
$f_H(\n,H)$. On a polygonal or triangulated interface, the normal is
discontinuous across elements, so $f_H(\n,H)$ generally has distinct
traces at a shared vertex even when the curvature field is continuous.
The discrete force and curvature equations must handle these traces
consistently to retain the energy law.

The main contributions of this work are as follows.
\begin{enumerate}
  \renewcommand{\labelenumi}{(\roman{enumi})}
  \setlength{\itemsep}{0.25em}
  \setlength{\parsep}{0pt}
  \interlinepenalty=10000
\item We develop a unified, fully discrete PFEM for closed curves and
surfaces with coupled orientation and curvature dependence. Using
continuous piecewise linear elements, the framework treats the $L^2$
flow, curve or surface diffusion, and the area- or volume-constrained
$L^2$ flow, including energies with orientation-dependent bending
rigidity and spontaneous curvature.
\item Under \cref{ass:density}, we prove unconditional discrete energy
dissipation for all three flows, together with exact conservation of
the enclosed area or volume for diffusion and the constrained $L^2$
flow. The analysis allows nonseparable densities and non-even
anisotropies, without requiring joint convexity in orientation and
curvature.
\item Numerical experiments in two and three dimensions exhibit
approximately second-order convergence in the manifold distance,
confirm the discrete structural properties, and illustrate shape
relaxation under anisotropic surface and bending energies.
\end{enumerate}

The remainder of the paper is organized as follows.
Section~\ref{sec:geometry} introduces the geometric identities and weak
formulations. Section~\ref{sec:matrices} recalls the local energy estimate
and states the assumptions on the density. The finite element scheme and
its energy analysis are presented in Sections~\ref{sec:discrete}
and~\ref{sec:energy-proof}, respectively. Section~\ref{sec:volume}
treats the area- and volume-preserving flows. Numerical results are
reported in Section~\ref{sec:numerical}, followed by concluding remarks
in Section~\ref{sec:conclusions}.

\section{Geometric setting and weak formulation}\label{sec:geometry}

\subsection{The surface energy}
Let $\Gamma(t)\subset\R^d$, $d=2,3$, be a smooth, closed, oriented hypersurface: a planar curve for $d=2$ and a surface for $d=3$. It is parametrized by $\X(\cdot,t)\colon\Gamma_0\to\Gamma(t)$. We write $\n$ for its outward unit normal and $\vvec=\dt\X$ for the material velocity. The derivative $\dt$ is taken at a fixed point of the reference surface. In particular, $\vvec$ may have a tangential component.

We use the row convention for the gradient of a vector field: $(\grad\wvec)_{ij}=D_jw_i$, where $D_j$ is the $j$th component of the surface gradient. Set
\begin{equation}\label{eq:geometry}
  \Pmat=I_d-\n\n^T=\grad\X,\qquad
  \Amat=\grad\n,\qquad H=\tr\Amat.
\end{equation}
The Weingarten matrix $\Amat$ is symmetric and satisfies $\Amat\n=0$. Our convention for $H$ is the sum of the principal curvatures, reducing to the signed curvature for a planar curve. Thus $H=(d-1)/R$ on an outward-oriented circle or sphere of radius $R$. For matrices, $U:V=\tr(U^TV)$ and $|U|^2=U:U$. Scalar and vector $L^2(\Gamma)$ inner products are denoted by $(\cdot,\cdot)_\Gamma$, and the corresponding matrix inner product by $\langle\cdot,\cdot\rangle_\Gamma$.

For the energy \eqref{eq:intro-energy}, we take a density
$f\colon\sphere^{d-1}\times\R\to\R_{>0}$.
Throughout the continuous derivation, the surface and the density are
sufficiently smooth for the stated derivatives and integrations by parts.
Convexity in $H$ is required only for the discrete energy estimate.

We denote the derivative with respect to the second argument by $f_H$ and introduce
\begin{equation}\label{eq:density-derivatives}
  \avec(\n,H)=\nabla_{\sphere^{d-1}}f(\n,H),\qquad
  r=f_H(\n,H).
\end{equation}
Here $H$ is held fixed when taking the spherical gradient. Equivalently, $\avec=\Pmat\nabla_{\bs p}\widetilde f(\n,H)$ for any smooth extension $\widetilde f$ off the unit sphere. This definition is independent of the extension and gives $\avec\cdot\n=0$. The dependence on curvature in \eqref{eq:intro-energy} is through $H$; an independent dependence on Gaussian curvature would require additional terms.

\subsection{An evolution identity for the curvature}
The normal and the surface measure satisfy
\begin{subequations}\label{eq:transport}
\begin{align}
  \dt\n&=-(\grad\vvec)^T\n,\label{eq:normal-transport}\\
  \frac{\mathrm d}{\mathrm dt}\int_{\Gamma(t)}u\dd A
  &=\int_{\Gamma(t)}\bigl(\dt u+u\Pmat:\grad\vvec\bigr)\dd A.
  \label{eq:area-transport}
\end{align}
\end{subequations}
The first identity follows by differentiating the orthogonality of $\n$ to the coordinate tangent vectors, together with $|\n|^2=1$.

For a vector field $\wvec$ and a scalar function $\psi$, define
\begin{equation}\label{eq:B-continuous}
  \B_\Gamma(\wvec,\psi)
  \coloneqq\bigl\langle\n(\grad\psi)^T-\psi\Amat,\grad\wvec\bigr\rangle_\Gamma.
\end{equation}
The same form will appear in the curvature equation and in the first variation of the energy.

\begin{lemma}[Curvature evolution]\label{lem:curvature}
For a smooth surface evolution and every smooth scalar function $\psi$,
\begin{equation}\label{eq:curvature-weak}
  (\dt H,\psi)_\Gamma=\B_\Gamma(\vvec,\psi).
\end{equation}
\end{lemma}
\begin{proof}
In local coordinates, let $g_{\alpha\beta}=\partial_\alpha\X\cdot\partial_\beta\X$ and let $(g^{\alpha\beta})$ be the inverse metric. Differentiating $H=g^{\alpha\beta}\partial_\alpha\n\cdot\partial_\beta\X$, with repeated Greek indices summed over $1,\ldots,d-1$, gives
\begin{equation}\label{eq:curvature-strong}
  \dt H=-\grad\cdot\bigl((\grad\vvec)^T\n\bigr)
          -\Amat:\grad\vvec.
\end{equation}
Indeed, the derivative of the inverse metric contributes $-2\Amat:\grad\vvec$, while differentiating the last tangent vector contributes $\Amat:\grad\vvec$. The remaining term is $\grad\cdot\dt\n$. The vector $(\grad\vvec)^T\n$ is tangential. Multiplication by $\psi$ and integration by parts on the closed surface prove \eqref{eq:curvature-weak}.
\end{proof}

The identity \eqref{eq:curvature-weak} is used for Willmore flow in \cite{bao2025energy3D}. For a purely normal velocity it reduces to
\begin{equation}\label{eq:normal-curvature}
  \dt H=-\lap V_n-|\Amat|^2V_n,
  \qquad V_n=\vvec\cdot\n.
\end{equation}
The weak form \eqref{eq:curvature-weak} also allows tangential motion and requires only first spatial derivatives of its arguments.

\subsection{First variation and chemical potential}
The chain rule and \eqref{eq:transport} give
\begin{equation}\label{eq:energy-chain}
  \frac{\mathrm d W}{\mathrm dt}
  =(r,\dt H)_\Gamma
   +\bigl\langle f\Pmat-\n\avec^T,\grad\vvec\bigr\rangle_\Gamma
  =\bigl\langle\bs S,\grad\vvec\bigr\rangle_\Gamma,
\end{equation}
where $f$ and $\avec$ are evaluated at $(\n,H)$ and, by \cref{lem:curvature},
\begin{equation}\label{eq:stress}
  \bs S=f\Pmat-\n\avec^T-r\Amat+\n(\grad r)^T.
\end{equation}
Notice that $\bs S\n=0$. Thus the derivative index of the stress is tangential, although the stress itself need not be symmetric.

\begin{proposition}[First variation]\label{prop:first-variation}
The first variation of \eqref{eq:intro-energy} is
\begin{equation}\label{eq:first-variation}
  \delta W[\om]=(\mu\n,\om)_\Gamma
   =\langle\bs S,\grad\om\rangle_\Gamma,
\end{equation}
with chemical potential
\begin{equation}\label{eq:chemical-potential}
  \mu=\grad\cdot\avec-\lap f_H+fH-f_H|\Amat|^2.
\end{equation}
\end{proposition}
\begin{proof}
For a normal variation $\om=\eta\n$, we have
$\grad(\eta\n)=\n(\grad\eta)^T+\eta\Amat$. Substituting this expression in \eqref{eq:energy-chain} yields
\[
  \delta W[\eta\n]
  =\int_\Gamma\left[-\avec\cdot\grad\eta
        +\grad r\cdot\grad\eta
        +(fH-r|\Amat|^2)\eta\right]\dd A.
\]
Integration by parts gives \eqref{eq:chemical-potential}. A tangential variation is an infinitesimal reparametrization and has zero first variation. Decomposing an arbitrary $\om$ into its normal and tangential parts proves \eqref{eq:first-variation}.
\end{proof}

This is the mean-curvature-dependent case of the general first-variation formula in \cite{dogan2012first}. If $f=\gamma(\n)$, then \eqref{eq:chemical-potential} gives the anisotropic mean curvature. If $f=H^2/2$, it gives
\begin{equation}\label{eq:willmore}
  \mu=-\lap H+\tfrac12H^3-H|\Amat|^2.
\end{equation}
These two cases also fix the signs in the flow laws below.

\subsection{Weak formulations of the gradient flows}
We consider the $L^2$ gradient flow and surface diffusion ($H^{-1}$ gradient flow) of the energy \eqref{eq:intro-energy}. The normal velocity $V_n=\vvec\cdot\n$ satisfies,
\begin{equation}\label{eq:flow-laws}
  V_n=-\mu,\qquad\text{and}\qquad V_n=\lap\mu,
\end{equation}
respectively. For a nondegenerate dependence on $H$, these are fourth- and sixth-order geometric equations. A formulation with first derivatives is obtained by retaining $H$ and $r$ as separate scalar fields.

For the $L^2$ flow, we seek $(\X,\mu,H,r)$ such that
\begin{subequations}\label{eq:continuous-system}
\begin{align}
  (\vvec\cdot\n,\varphi)_\Gamma
     &=-(\mu,\varphi)_\Gamma,\label{eq:continuous-velocity}\\
  (\mu\n,\om)_\Gamma
     &=\langle f\Pmat-\n\avec^T,\grad\om\rangle_\Gamma
       +\B_\Gamma(\om,r),\label{eq:continuous-force}\\
  (\dt H,\psi)_\Gamma&=\B_\Gamma(\vvec,\psi),\label{eq:continuous-H}\\
  (r,\zeta)_\Gamma&=(f_H(\n,H),\zeta)_\Gamma,\label{eq:continuous-r}
\end{align}
\end{subequations} for all $\varphi,\psi,\zeta\in H^1(\Gamma)$ and $\om\in [H^1(\Gamma)]^d$. 

At $t=0$ we prescribe $\X_0$ and the geometric curvature $H_0=-\n_0\cdot\Delta_{\Gamma_0}\X_0$. For surface diffusion, only \eqref{eq:continuous-velocity} changes:
\begin{equation}\label{eq:continuous-diffusion}
  (\vvec\cdot\n,\varphi)_\Gamma
     =-(\grad\mu,\grad\varphi)_\Gamma.
\end{equation}

\begin{proposition}[Continuous dissipation and conservation]\label{prop:continuous-energy}
Smooth geometric solutions of \eqref{eq:continuous-system} satisfy
\begin{equation}\label{eq:continuous-l2-energy}
  \frac{\mathrm dW}{\mathrm dt}=-\|\mu\|_{L^2(\Gamma)}^2.
\end{equation}
For surface diffusion, with \eqref{eq:continuous-velocity} replaced by \eqref{eq:continuous-diffusion},
\begin{equation}\label{eq:continuous-sd-energy}
  \frac{\mathrm dW}{\mathrm dt}=-\|\grad\mu\|_{L^2(\Gamma)}^2,
  \qquad \frac{\mathrm dV}{\mathrm dt}=0,
\end{equation}
where $V$ is the enclosed area for $d=2$ and volume for $d=3$. We use the term volume for this enclosed $d$-dimensional measure in the unified notation below.
\end{proposition}
\begin{proof}
Take $\om=\vvec$, $\psi=r$, and $\zeta=\dt H$ in the last three equations of \eqref{eq:continuous-system}. Equation \eqref{eq:energy-chain} then gives $\mathrm dW/\mathrm dt=(\mu\n,\vvec)_\Gamma$. Testing the appropriate velocity equation with $\varphi=\mu$ proves the energy identities. For surface diffusion, the test $\varphi=1$ gives $(\vvec\cdot\n,1)_\Gamma=0$, which is the volume transport formula.
\end{proof}

The role of the last equation in \eqref{eq:continuous-system} is simple at this level: it states $r=f_H$. Its separate weak form becomes useful after discretization, because the normal-dependent expression $f_H$ need not be a continuous finite element function.

\section{The stabilizing function and the local energy estimate}\label{sec:matrices}

At fixed curvature, $f(\n,H)$ is a normal-dependent surface energy. We apply the symmetric surface energy matrix to this density, retaining the curvature as a parameter. The stabilizing function is therefore a function of both arguments of $f$.

\subsection{Conditions on the density and the surface energy matrix}
Let $I\subseteq\R$ be an interval containing the curvature values under consideration. The following conditions will be used in the stability analysis.

\begin{assumption}[Energy density]\label{ass:density}
For every $s\in I$, the function $f(\cdot,s)$ is positive and belongs to $C^2(\sphere^{d-1})$. Moreover,
\begin{equation}\label{eq:direction-condition}
  (5-d)f(\n,s)-f(-\n,s)\geq0,
  \qquad (\n,s)\in\sphere^{d-1}\times I,\qquad d=2,3.
\end{equation}
For each $\n\in\sphere^{d-1}$, the function $s\mapsto f(\n,s)$ is continuously differentiable and convex on $I$.
\end{assumption}

The directional condition is automatic for a positive even density. It neither requires evenness nor imposes convexity on the one-homogeneous extension in the normal variable. The curvature dependence need not separate from the directional dependence. Only positivity, directional regularity, and \eqref{eq:direction-condition} are needed for the local estimate in this section; convexity in $s$ enters the time-discrete energy proof.

For fixed $s$, define the one-homogeneous extension and the Cahn--Hoffman vector by
\begin{equation}\label{eq:homogeneous}
  F(\bs p,s)=|\bs p|f\left(\frac{\bs p}{|\bs p|},s\right),\qquad
  \bs\xi(\n,s)=\nabla_{\bs p}F(\n,s)=f(\n,s)\n+\avec(\n,s).
\end{equation}
Given a nonnegative stabilizing function $k\colon\sphere^{d-1}\times I\to\R$, set
\begin{equation}\label{eq:Zk}
  \Z_k(\n,s)=f(\n,s)I_d-\n\bs\xi(\n,s)^T
             -\bs\xi(\n,s)\n^T+k(\n,s)\n\n^T.
\end{equation}
This is the symmetric branch of the surface energy matrices in \cite{bao2026alpha,liZhangSurface3D}; see also \cite{bao2023symmetrized3D,li2025optimal}. Since $\Pmat\n=0$,
\begin{equation}\label{eq:matrix-consistency}
  \Z_k(\n,s)\Pmat=f(\n,s)\Pmat-\n\avec(\n,s)^T.
\end{equation}
Consequently, \eqref{eq:continuous-force} can be written as
\begin{equation}\label{eq:matrix-weak-force}
  (\mu\n,\om)_\Gamma
    =\langle\Z_k(\n,H)\grad\X,\grad\om\rangle_\Gamma
      +\B_\Gamma(\om,r).
\end{equation}
The stabilizing term leaves this continuous identity unchanged. Its size matters in the discrete estimate, where the normals and measures of the old and new elements differ.

\subsection{The local energy estimate}

\begin{lemma}[Local energy estimate]\label{lem:local-energy}
Fix $s\in I$. Suppose that $f(\cdot,s)>0$, $f(\cdot,s)\in C^2(\sphere^{d-1})$, and \eqref{eq:direction-condition} holds. There is a finite nonnegative threshold $k_0(\n,s)$ such that the following estimate holds whenever $k(\n,s)\geq k_0(\n,s)$.

Let $\sigma$ be an oriented nondegenerate $(d-1)$-simplex of a polygonal curve or polyhedral surface $\Gamma$, and let $\X$ map $\sigma$ affinely onto another such simplex $\bar\sigma$, preserving the ordering of corresponding vertices. Denote their induced unit normals by $\n$ and $\bar\n$, and let $\id$ be the identity map. With the elementwise surface gradient defined in \eqref{eq:discrete-gradient-2D}--\eqref{eq:vector-discrete-gradient}, we have
\begin{equation}\label{eq:local-estimate}
  |\sigma|\bigl(\Z_k(\n,s)\grad\X|_\sigma\bigr):
       \bigl(\grad\X|_\sigma-\grad\id|_\sigma\bigr)
  \geq |\bar\sigma|f(\bar\n,s)-|\sigma|f(\n,s).
\end{equation}
\end{lemma}
\begin{proof}
For fixed $s$, set $\gamma_s(\n)=f(\n,s)$ and apply the local energy estimate for the symmetric surface energy matrix. For $d=2$, this is the estimate in \cite{li2025optimal}; see also the branch $\alpha=-1$ in \cite{bao2026alpha}. For $d=3$, the proof is the same as that of \cite[Lemma~8.2]{zhangThesis}, using the directional stabilization result of \cite[Theorem~5.1]{liZhangSurface3D}. The stated assumptions give the required conditions on $\gamma_s$, and \eqref{eq:Zk} is its symmetric surface energy matrix.
\end{proof}

The threshold may depend on both the normal and the curvature. Throughout the paper, we choose
\begin{equation}\label{eq:admissible-k}
  k(\n,s)\geq k_0(\n,s),\qquad (\n,s)\in\sphere^{d-1}\times I.
\end{equation}
The lemma gives a finite threshold at each fixed $s$; it does not require a uniform bound over all $s\in\R$. In the discrete scheme, this condition will be imposed at the new curvature values. Thus the local estimate follows from the conditions on $f$ and the choice of $k$, rather than being assumed separately.

\begin{remark}[Densities that vanish]\label{rem:zero-density}
The positivity in \cref{lem:local-energy} is needed to apply the directional stabilization result. A density that vanishes at some curvature values can still be treated if its local estimate is checked there. For example, if $f(\n,s)=\Phi(s)\geq0$ is independent of $\n$, choosing $k(\n,s)=2\Phi(s)$ gives $\Z_k=\Phi(s)I_d$. The usual segment length or triangle area inequality gives \eqref{eq:local-estimate}, including $\Phi(s)=0$. The energy proof below then applies whenever $\Phi$ is differentiable and convex. In particular, the Willmore density $\Phi(s)=s^2/2$ is included.
\end{remark}

\section{A parametric finite element discretization}\label{sec:discrete}

\subsection{Discrete surfaces and quadrature}
Let $0=t_0<t_1<\cdots$ with $\tau_m=t_{m+1}-t_m$. We approximate $\Gamma(t_m)$ by a closed polygonal curve or triangulated surface
\begin{equation}\label{eq:triangulation}
  \Gamma^m=\bigcup_{j=1}^J\overline{\sigma_j^m},\qquad
  \sigma_j^m=[\q_{j_1}^m,\cdots,\q_{j_d}^m].
\end{equation}
The elements are nondegenerate segments for $d=2$ and triangles for $d=3$,
with vertex orderings chosen consistently with the outward orientation.
Their associated normal vectors are
\begin{equation}\label{eq:face-normal}
  \mathcal J\{\sigma_j\}\coloneqq
  \begin{cases}
    (\q_{j_2}-\q_{j_1})^\perp, & d=2,\\
    (\q_{j_2}-\q_{j_1})\times(\q_{j_3}-\q_{j_1}), & d=3,
  \end{cases}
\end{equation}
where $(v_1,v_2)^\perp=(v_2,-v_1)$ is the clockwise rotation in two
dimensions. The element measure and unit normal are
\begin{equation}
  |\sigma_j|=\frac{|\mathcal J\{\sigma_j\}|}{d-1},\qquad
  \n_j=\frac{\mathcal J\{\sigma_j\}}{|\mathcal J\{\sigma_j\}|}.
\end{equation}

The continuous piecewise linear space is
\begin{equation}\label{eq:FE-space}
  \K^m=\K(\Gamma^m)\coloneqq\left\{u\in C(\Gamma^m):u|_{\sigma_j^m}\in\mathcal P^1(\sigma_j^m),
           \ 1\leq j\leq J\right\},
\end{equation} where $\mathcal P^1(\sigma_j^m)$ is the space of linear polynomials on $\sigma_j^m$. 

For a segment $\sigma=[\q_1,\q_2]$ in 2D, the discrete surface gradient is defined by
\begin{equation}\label{eq:discrete-gradient-2D}
  \left.\nabla_{\Gamma^m}u\right|_\sigma
  \coloneqq \left(u(\q_2)-u(\q_1)\right)\frac{\q_2-\q_1}{|\q_2-\q_1|^2},\qquad u\in\mathcal P^1(\sigma).
\end{equation}
For a triangle $\sigma=[\q_1,\q_2,\q_3]$ and $u\in\mathcal P^1(\sigma)$ in three dimensions, it is
\begin{equation}\label{eq:discrete-gradient}
  \left.\nabla_{\Gamma^m}u\right|_\sigma
  \coloneqq u(\q_1)\frac{(\q_2-\q_3)\times\n}{|\mathcal J\{\sigma\}|}
    +u(\q_2)\frac{(\q_3-\q_1)\times\n}{|\mathcal J\{\sigma\}|}
    +u(\q_3)\frac{(\q_1-\q_2)\times\n}{|\mathcal J\{\sigma\}|},
\end{equation}
where $\n=\mathcal J\{\sigma\}/|\mathcal J\{\sigma\}|$ is the unit normal to $\sigma$. For a vector-valued function $\wvec=(w_1,w_2,\cdots,w_d)^T$, its discrete surface Jacobian is
\begin{equation}\label{eq:vector-discrete-gradient}
  \left.\nabla_{\Gamma^m}\wvec\right|_\sigma
  \coloneqq
  \begin{bmatrix}
      (\nabla_{\Gamma^m}w_1)^T\\
      (\nabla_{\Gamma^m}w_2)^T\\
      \vdots\\
      (\nabla_{\Gamma^m}w_d)^T
     \end{bmatrix},
  \qquad \forall\,\wvec\in[\mathcal P^1(\sigma)]^d.
\end{equation}
Both are constant on each simplex, and $\nabla_{\Gamma^m}\id=I_d-\n^m\otimes\n^m$.

We regard $\X^m=\id$ as a map on $\Gamma^m$, while $\X^{m+1}\in[\K^m]^d$ maps $\q_i^m$ to $\q_i^{m+1}$. We denote
\begin{equation}\label{eq:increments}
  \dX=\X^{m+1}-\X^m,\qquad \dH=H^{m+1}-H^m.
\end{equation}
The scalar unknowns at level $m+1$ are first represented on $\Gamma^m$. After the step, their nodal coefficients are carried to the corresponding vertices of $\Gamma^{m+1}$. The notation $H_i^{m+1}$ denotes this same coefficient on either surface.

For piecewise continuous functions with well-defined vertex traces, we use the mass-lumped inner product, as in \cite{liZhangSurface3D}:
\begin{equation}\label{eq:lumping}
  (u,v)_{\Gamma^m}^h
  \coloneqq\frac1d\sum_{j=1}^J\sum_{\ell=1}^d|\sigma_j^m|
     u((\q_{j_\ell}^m)^-)v((\q_{j_\ell}^m)^-).
\end{equation}
Here the value at each vertex is taken from the interior of the element:
\[
  u((\q_{j_\ell}^m)^-)
    \coloneqq\lim_{\substack{\q\to\q_{j_\ell}^m\\\q\in\sigma_j^m}}u(\q)
    \eqqcolon u_{j\ell}.
\]
We write $\|u\|_{m,h}^2=(u,u)_{\Gamma^m}^h$. For vector-valued functions, the product in \eqref{eq:lumping} is replaced by the Euclidean dot product. For matrix-valued functions, we set
\[
  \langle\bs U,\bs V\rangle_{\Gamma^m}^h
  \coloneqq\frac1d\sum_{j=1}^J\sum_{\ell=1}^d|\sigma_j^m|
       \bs U((\q_{j_\ell}^m)^-):\bs V((\q_{j_\ell}^m)^-).
\]
Each element therefore contributes the diagonal nodal mass matrix
$\tfrac1d|\sigma_j^m|I_d$. All inner products in the discrete schemes use this vertex quadrature. For elementwise constant integrands, including products of gradients of piecewise linear functions, it agrees with exact integration.

All nonlinear coefficients are evaluated at these element vertex traces. For example,
\[
  f(\n^m,H^{m+1})_{j\ell}=f(\n_j^m,H_{j_\ell}^{m+1}).
\]
The curvature coefficient is shared by adjacent elements, but their normals are different. Consequently, this expression generally has more than one trace at a vertex.

\subsection{A discrete curvature evolution form}
We approximate the Weingarten matrix from a continuous recovered normal. Let $\mathcal T_i^m$ be the set of indices of elements incident to vertex $i$. At each vertex, set
\begin{equation}\label{eq:vertex-normal}
  \N_i^m=
    \frac{\displaystyle\sum_{j\in\mathcal T_i^m}|\sigma_j^m|\n_j^m}
         {\displaystyle\left|\sum_{j\in\mathcal T_i^m}|\sigma_j^m|\n_j^m\right|},
\end{equation}
assuming the denominator is nonzero, and interpolate these values in $[\K^m]^d$. On each element define
\begin{equation}\label{eq:discrete-Weingarten}
  \Amat^m=\Pmat^m\sym(\nabla_{\Gamma^m}\N^m)\Pmat^m,
  \qquad \sym U=\tfrac12(U+U^T).
\end{equation}
This gives a symmetric, tangential, piecewise constant matrix. The recovered normal is used only in this construction; the density and energy matrix use the face normal \eqref{eq:face-normal}.

The discrete counterpart of \eqref{eq:B-continuous} is
\begin{equation}\label{eq:B-discrete}
  \B_m(\wvec,\psi)
   =\bigl\langle\n^m(\nabla_{\Gamma^m}\psi)^T-\psi\Amat^m,
                      \nabla_{\Gamma^m}\wvec\bigr\rangle_{\Gamma^m}^h.
\end{equation}
Other approximations of $\Amat$ can be used in this form. The energy proof requires that the same $\B_m$ appear in the force and curvature equations. The choice \eqref{eq:discrete-Weingarten} makes the method fully specified.

Choose a stabilizing function $k$ satisfying \eqref{eq:admissible-k}. For a fixed $H\in\K^m$ with nodal values in $I$, define
\begin{equation}\label{eq:D-discrete}
  \D_m^k(\Y,\om;H)
   =\bigl\langle\Z_k(\n^m,H)\nabla_{\Gamma^m}\Y,\nabla_{\Gamma^m}\om\bigr\rangle_{\Gamma^m}^h.
\end{equation}
This form is bilinear in $\Y$ and $\om$ at fixed $H$. Its coefficient includes $k(\n^m,H)$, so the stabilizing function retains its curvature dependence.

\subsection{The fully discrete \texorpdfstring{$L^2$}{L2} flow}
Given $\Gamma^m$ and $H^m\in\K^m$, find
\[
  (\X^{m+1},\mu^{m+1},H^{m+1},r^{m+1})
       \in[\K^m]^d\times(\K^m)^3
\]
such that, for every $\varphi,\psi,\zeta\in\K^m$ and $\om\in[\K^m]^d$,
\begin{subequations}\label{eq:scheme}
\begin{align}
  \left(\frac{\dX}{\tau_m}\cdot\n^m,\varphi\right)_{\Gamma^m}^h
     &=-(\mu^{m+1},\varphi)_{\Gamma^m}^h,\label{eq:scheme-velocity}\\
  (\mu^{m+1}\n^m,\om)_{\Gamma^m}^h
     &=\D_m^k(\X^{m+1},\om;H^{m+1})
        +\B_m(\om,r^{m+1}),\label{eq:scheme-force}\\
  (\dH,\psi)_{\Gamma^m}^h
     &=\B_m(\dX,\psi),\label{eq:scheme-H}\\
  (r^{m+1},\zeta)_{\Gamma^m}^h
     &=(f_H(\n^m,H^{m+1}),\zeta)_{\Gamma^m}^h.\label{eq:scheme-r}
\end{align}
\end{subequations}
Equation \eqref{eq:scheme-H} is written for increments, so it contains no additional factor $\tau_m$. The old geometry is used in $\B_m$, whereas the new curvature enters both $\Z_k$ and $f_H$, including the value of $k$. The resulting system is nonlinear even though all four finite element fields are piecewise linear.

In particular, the stabilizing function is evaluated at the same element vertex traces as the density:
\begin{equation}\label{eq:nodal-k}
  k(\n_j^m,H_{j_\ell}^{m+1})
       \geq k_0(\n_j^m,H_{j_\ell}^{m+1}),
  \qquad 1\leq j\leq J,\quad 1\leq\ell\leq d.
\end{equation}
The function may be chosen differently at different time steps, provided this bound holds. Freezing it at the old curvature is justified only if the resulting value also bounds the threshold at the new curvature.

The last equation is a mass-lumped projection. Let
\begin{equation}\label{eq:nodal-mass}
  M_i^m=\frac1d\sum_{j\in\mathcal T_i^m}|\sigma_j^m|>0.
\end{equation}
Testing \eqref{eq:scheme-r} with the $i$th nodal basis function gives
\begin{equation}\label{eq:r-nodal}
  r_i^{m+1}=\frac{1}{dM_i^m}
       \sum_{j\in\mathcal T_i^m}|\sigma_j^m|
                f_H(\n_j^m,H_i^{m+1}).
\end{equation}
Hence $r^{m+1}$ can be eliminated by a nodal calculation for a general density $f$.

Keeping the projection in the variational statement is nevertheless useful. It ensures that $r^{m+1}$ is an admissible test function in \eqref{eq:scheme-H} and gives the exact identity
\begin{equation}\label{eq:projection-pairing}
  (r^{m+1},\dH)_{\Gamma^m}^h
     =(f_H(\n^m,H^{m+1}),\dH)_{\Gamma^m}^h.
\end{equation}
There is therefore no need to differentiate the generally discontinuous expression $f_H(\n^m,H^{m+1})$.

\section{Discrete energy dissipation}\label{sec:energy-proof}

The energy of the discrete surface and its curvature field is
\begin{equation}\label{eq:discrete-energy}
  W^m=\frac1d\sum_{j=1}^J|\sigma_j^m|
                       \sum_{\ell=1}^df(\n_j^m,H_{j_\ell}^m).
\end{equation}
Both the surface measure and the density are evaluated on the current surface. The quadrature is part of the definition of $W^m$.

\begin{theorem}[Energy dissipation for the $L^2$ flow]\label{thm:L2-energy}
Let $f$ satisfy \cref{ass:density}, and choose the stabilizing function so that \eqref{eq:nodal-k} holds. If \eqref{eq:scheme} has a solution with nondegenerate old and new elements and with $H_i^m,H_i^{m+1}\in I$, then
\begin{equation}\label{eq:L2-energy}
  W^{m+1}+\tau_m\|\mu^{m+1}\|_{m,h}^2\leq W^m.
\end{equation}
This estimate holds for every $\tau_m>0$.
\end{theorem}
\begin{proof}
Choose $\varphi=\mu^{m+1}$ in \eqref{eq:scheme-velocity}, $\om=\dX$ in \eqref{eq:scheme-force}, and $\psi=r^{m+1}$ in \eqref{eq:scheme-H}. The resulting identities give
\begin{equation}\label{eq:discrete-balance}
  \D_m^k(\X^{m+1},\dX;H^{m+1})+(\dH,r^{m+1})_{\Gamma^m}^h
       =-\tau_m\|\mu^{m+1}\|_{m,h}^2.
\end{equation}

To estimate its left-hand side, introduce the energy with new curvature coefficients on the old geometry,
\begin{equation}\label{eq:intermediate-energy}
  \widehat W^{m+1|m}
      =\frac1d\sum_{j=1}^J|\sigma_j^m|
                         \sum_{\ell=1}^df(\n_j^m,H_{j_\ell}^{m+1}).
\end{equation}
By \eqref{eq:nodal-k}, \cref{lem:local-energy} applies to the map $\X^{m+1}|_{\sigma_j^m}$ with $s=H_{j_\ell}^{m+1}$. Divide each inequality by $d$ and sum over $j$ and $\ell$. Since $\nabla_{\Gamma^m}\X^m=\Pmat^m$, this yields
\begin{equation}\label{eq:geometry-bound}
  \D_m^k(\X^{m+1},\dX;H^{m+1})
        \geq W^{m+1}-\widehat W^{m+1|m}.
\end{equation}

Convexity in the second argument gives, for any $a,b\in I$,
\begin{equation}\label{eq:convexity}
  f_H(\n,b)(b-a)\geq f(\n,b)-f(\n,a).
\end{equation}
At each old element vertex, take $a=H_{j_\ell}^m$ and $b=H_{j_\ell}^{m+1}$. After summation and use of \eqref{eq:projection-pairing}, we obtain
\begin{equation}\label{eq:curvature-bound}
  (\dH,r^{m+1})_{\Gamma^m}^h
      \geq\widehat W^{m+1|m}-W^m.
\end{equation}
Adding \eqref{eq:geometry-bound} and \eqref{eq:curvature-bound} cancels the intermediate energy. Substitution in \eqref{eq:discrete-balance} proves \eqref{eq:L2-energy}.
\end{proof}

The argument separates the two changes in the density without separating its two arguments. In particular, it does not require $f-Hf_H$ to be nonnegative. The local matrix estimate controls the normal and area, while the implicit curvature derivative controls the change of $H$.

\begin{corollary}[Accumulated dissipation]\label{cor:cumulative}
Under the hypotheses of \cref{thm:L2-energy}, a sequence of discrete solutions satisfies
\begin{equation}\label{eq:cumulative-energy}
  W^M+\sum_{m=0}^{M-1}\tau_m\|\mu^{m+1}\|_{m,h}^2\leq W^0.
\end{equation}
\end{corollary}
\begin{proof}
Sum \eqref{eq:L2-energy} over $m=0,\ldots,M-1$.
\end{proof}

The same proof applies to a density covered by \cref{rem:zero-density}, or to a matrix for which the local estimate has been verified directly.

\begin{remark}\label{rem:scope}
The absence of a time-step restriction in \eqref{eq:L2-energy} concerns the energy estimate. The theorem assumes a solution of the nonlinear system and nondegenerate elements. It does not establish solvability, prevent mesh degeneration, or estimate the difference between the evolved field $H^m$ and the geometric curvature of an underlying smooth surface.
\end{remark}

\section{Area- and volume-preserving flows}\label{sec:volume}

To preserve the enclosed area or volume, we use the time-averaged
normal introduced by Bao and Zhao \cite{bao2021structure}.

\subsection{The discrete volume increment}
On each element $\sigma_j^m$, the piecewise constant vector
$\n^{m+\frac12}$ is defined by
\begin{equation}\label{eq:average-normal}
  \n^{m+\frac12}\big|_{\sigma_j^m}
  \coloneqq\frac{\mathcal J\{\sigma_j^m\}
             +4\mathcal J\{\sigma_j^{m+\frac12}\}
             +\mathcal J\{\sigma_j^{m+1}\}}
         {6|\mathcal J\{\sigma_j^m\}|},
\end{equation}
where
\[
  \sigma_j^{m+\frac12}
    \coloneqq\left[
       \frac{\q_{j_1}^m+\q_{j_1}^{m+1}}2,
       \ldots,
       \frac{\q_{j_d}^m+\q_{j_d}^{m+1}}2
     \right]
\]
is the midpoint element and $\mathcal J\{\sigma\}$ is the orientation
vector defined in \eqref{eq:face-normal}.

The oriented enclosed measure is
\begin{equation}\label{eq:volume}
  V^m=\frac1d(\id,\n^m)_{\Gamma^m}^h
    =\frac{1}{d^2}\sum_{j=1}^J|\sigma_j^m|
                 \sum_{\ell=1}^d\q_{j_\ell}^m\cdot\n_j^m.
\end{equation}
For an embedded mesh with outward orientation, this is its enclosed area
in two dimensions and volume in three dimensions, by the divergence theorem.

\begin{lemma}[Exact volume increment \cite{bao2021structure}]\label{lem:volume}
Closed, consistently oriented polygonal curves or triangular surfaces with the same connectivity satisfy
\begin{equation}\label{eq:volume-increment}
  V^{m+1}-V^m=(\dX,\n^{m+\frac12})_{\Gamma^m}^h.
\end{equation}
\end{lemma}

\subsection{Surface diffusion}
We retain the curvature and projection equations \eqref{eq:scheme-H} and \eqref{eq:scheme-r} and replace the first two equations of \eqref{eq:scheme} by
\begin{subequations}\label{eq:sd-scheme}
\begin{align}
  \left(\frac{\dX}{\tau_m}\cdot\n^{m+\frac12},\varphi\right)_{\Gamma^m}^h
      &=-(\nabla_{\Gamma^m}\mu^{m+1},\nabla_{\Gamma^m}\varphi)_{\Gamma^m}^h,
        \label{eq:sd-velocity}\\
  (\mu^{m+1}\n^{m+\frac12},\om)_{\Gamma^m}^h
      &=\D_m^k(\X^{m+1},\om;H^{m+1})
         +\B_m(\om,r^{m+1}).\label{eq:sd-force}
\end{align}
\end{subequations}
The gradient product in \eqref{eq:sd-velocity} uses the same mass-lumped inner product. Since both gradients are constant on each element, this quadrature is exact. The same time-averaged normal is used in both equations. The two occurrences of $\B_m$ still use the old geometry.

\begin{theorem}[Volume conservation and energy dissipation]\label{thm:sd}
Let $f$ satisfy \cref{ass:density}, and choose $k$ according to \eqref{eq:nodal-k}. Then every solution of \eqref{eq:sd-scheme}, \eqref{eq:scheme-H}, and \eqref{eq:scheme-r} with nondegenerate old and new elements and nodal curvature values in $I$ satisfies
\begin{equation}\label{eq:sd-structure}
  V^{m+1}=V^m,\qquad
  W^{m+1}+\tau_m\|\nabla_{\Gamma^m}\mu^{m+1}\|_{L^2(\Gamma^m)}^2\leq W^m.
\end{equation}
\end{theorem}
\begin{proof}
Take $\varphi=1$ in \eqref{eq:sd-velocity}. Its right-hand side vanishes, so \cref{lem:volume} gives volume conservation. Next choose $\varphi=\mu^{m+1}$ in \eqref{eq:sd-velocity}, $\om=\dX$ in \eqref{eq:sd-force}, and $\psi=r^{m+1}$ in \eqref{eq:scheme-H}. We obtain
\begin{equation}\label{eq:sd-balance}
  \D_m^k(\X^{m+1},\dX;H^{m+1})+(\dH,r^{m+1})_{\Gamma^m}^h
    =-\tau_m\|\nabla_{\Gamma^m}\mu^{m+1}\|_{L^2(\Gamma^m)}^2.
\end{equation}
The left-hand side is bounded below by $W^{m+1}-W^m$ by \eqref{eq:geometry-bound} and \eqref{eq:curvature-bound}. This proves the energy inequality.
\end{proof}

The use of $\n^{m+\frac12}$ in the force equation preserves the pairing in \eqref{eq:sd-balance}. Using it only in the velocity equation would retain the volume identity but would no longer give this energy proof.

\subsection{The volume-constrained \texorpdfstring{$L^2$}{L2} flow}
The continuous volume-constrained flow has the velocity law
\begin{equation}\label{eq:constrained-continuous}
  V_n=-(\mu-\lambda),\qquad
  \lambda=\frac{(\mu,1)_\Gamma}{(1,1)_\Gamma}.
\end{equation}
It satisfies $\mathrm dV/\mathrm dt=0$ and
$\mathrm dW/\mathrm dt=-\|\mu-\lambda\|_{L^2(\Gamma)}^2$. Its discrete version uses \eqref{eq:sd-force}, \eqref{eq:scheme-H}, and \eqref{eq:scheme-r}, together with
\begin{equation}\label{eq:constrained-velocity}
  \left(\frac{\dX}{\tau_m}\cdot\n^{m+\frac12},\varphi\right)_{\Gamma^m}^h
      =-(\mu^{m+1}-\lambda^{m+1},\varphi)_{\Gamma^m}^h,
  \qquad
  \lambda^{m+1}=\frac{(\mu^{m+1},1)_{\Gamma^m}^h}{(1,1)_{\Gamma^m}^h}.
\end{equation}

\begin{corollary}\label{cor:constrained}
Under the assumptions of \cref{thm:sd}, the volume-constrained scheme satisfies
\begin{equation}\label{eq:constrained-structure}
  V^{m+1}=V^m,\qquad
  W^{m+1}+\tau_m\|\mu^{m+1}-\lambda^{m+1}\|_{m,h}^2\leq W^m.
\end{equation}
\end{corollary}
\begin{proof}
The test $\varphi=1$ gives volume conservation. For the energy estimate, use $\varphi=\mu^{m+1}$ and the identity
\[
  (\mu^{m+1}-\lambda^{m+1},\mu^{m+1})_{\Gamma^m}^h
      =\|\mu^{m+1}-\lambda^{m+1}\|_{m,h}^2.
\]
The remaining tests and estimates are those in the proof of \cref{thm:sd}.
\end{proof}

\section{Numerical results}\label{sec:numerical}

All energy densities used below satisfy \cref{ass:density}.

\subsection{Error and convergence rate}\label{sec:convergence}

We test convergence using smooth densities that couple the normal and
curvature. In all experiments, the initial nodal curvature is obtained
by evaluating the analytical curvature of the prescribed smooth curve
or surface at the mesh vertices.

Let $\Gamma_1,\Gamma_2$ be two polygonal curves or polyhedral surfaces, and let $\Omega_1,\Omega_2$ be their interiors, respectively. We measure the difference between them by the manifold distance
\begin{equation}\label{eq:curve-manifold-distance}
  M(\Gamma_1,\Gamma_2)\coloneqq
  \bigl|\Omega_1\mathbin{\triangle}\Omega_2\bigr|=2|\Omega_1\cup\Omega_2|-|\Omega_1|-|\Omega_2|,
\end{equation}
i.e. the area or volume of the symmetric difference of the two enclosed regions. For planar curves, the numerical error $e^h(t)$ is then defined by \begin{equation}
  e^h(t)\coloneqq M\bigl(\Gamma_{h,\tau}(t),\Gamma(t)\bigr).
\end{equation}

\subsubsection{Convergence for planar curves}
We first test the method for planar curves on an ellipse with a density that couples the normal direction and the curvature. For the curves, we approximate the exact solution $\Gamma(t)$ by the numerical reference $\Gamma_{h_e,\tau_e}(t)$ with $h_e=2^{-10}$ and $\tau_e=1.28h_e^2$. 

Let $\kappa$ denote the curvature. We take
\begin{equation}\label{eq:curve-coupled-density}
  f(\n,\kappa)
   =\sqrt{\n^T\bs G_0\n}
     +0.1\sqrt{\n^T\bs G_1\n}\,\kappa^2
     +\frac{0.025}{4}\kappa^4,
  \qquad \n\in\sphere^1,
\end{equation}
where
\[
  \bs G_0=\begin{pmatrix}1.4&0\\0&1\end{pmatrix},
  \qquad
  \bs G_1=\begin{pmatrix}1.8&0.25\\0.25&0.8\end{pmatrix}.
\]
The initial curve is parametrized by
\begin{equation}\label{eq:curve-initial-ellipse}
  \X_0(\rho)=\bigl(1.5\cos(2\pi\rho),\,\sin(2\pi\rho)\bigr)^T,
  \qquad 0\leq\rho\leq1.
\end{equation}
We sample the initial position at $\rho_i=i/N$, $0\leq i<N$. Thus $h=1/N$ is the parameter mesh size, and the largest initial edge length is $O(h)$. Both the $L^2$ flow and curve diffusion are evaluated at $t=0.1,0.2,0.3$ with $\tau=1.28h^2$. We take $h_0=2^{-5}$ and $\tau_0=1.28h_0^2$.

\begin{table}[!htbp]
\centering\small
\caption{Error and convergence rates for the $L^2$ flow of planar curves.}\label{tab:curve-coupled-l2}
\begin{tabular}{ccccccc}
\toprule
$(h,\tau)$ & $e^h(t=0.1)$ & order &  $e^h(t=0.2)$ & order &  $e^h(t=0.3)$ & order \\
\midrule
$(h_0,\tau_0)$ & 3.32e-2 & -- & 3.55e-2 & -- & 3.70e-2 & -- \\
$(\frac{h_0}{2},\frac{\tau_0}{4})$  & 8.30e-3 & 2.001 & 8.88e-3 & 1.999 & 9.27e-3 & 1.998 \\
$(\frac{h_0}{2^2},\frac{\tau_0}{4^2})$ &  2.05e-3 & 2.016 & 2.20e-3 & 2.016 & 2.29e-3 & 2.015 \\
$(\frac{h_0}{2^3},\frac{\tau_0}{4^3})$ & 4.89e-4 & 2.070 & 5.23e-4 & 2.070 & 5.46e-4 & 2.070 \\
\bottomrule
\end{tabular}
\end{table}

\begin{table}[!htbp]
\centering\small
\caption{Error and convergence rates for curve diffusion.}\label{tab:curve-coupled-diffusion}
\begin{tabular}{ccccccc}
\toprule
$(h,\tau)$ & $e^h(t=0.1)$ & order &  $e^h(t=0.2)$ & order &  $e^h(t=0.3)$ & order \\
\midrule
$(h_0,\tau_0)$ & 3.03e-2 & -- & 3.04e-2 & -- & 3.03e-2 & -- \\
$(\frac{h_0}{2},\frac{\tau_0}{4})$  & 7.57e-3 & 2.001 & 7.59e-3 & 1.999 & 7.59e-3 & 1.999 \\
$(\frac{h_0}{2^2},\frac{\tau_0}{4^2})$ &  1.88e-3 & 2.013 & 1.89e-3 & 2.010 & 1.88e-3 & 2.010 \\
$(\frac{h_0}{2^3},\frac{\tau_0}{4^3})$ & 4.49e-4 & 2.062 & 4.53e-4 & 2.056 & 4.53e-4 & 2.055 \\
\bottomrule
\end{tabular}
\end{table}

Tables~\ref{tab:curve-coupled-l2} and~\ref{tab:curve-coupled-diffusion} display the errors and convergence rates for the $L^2$ flow and curve diffusion, respectively. The results exhibit second-order convergence in the manifold distance.

\subsubsection{Convergence for surfaces}
For surfaces, we apply the following coupled energy density
\begin{equation}\label{eq:surface-coupled-density}
  f(\n,H)
   =\sqrt{\n^T\bs G_0\n}
     +\frac{0.05}{2}\sqrt{\n^T\bs G_1\n}\,H^2
     +\frac{0.002}{4}H^4,
  \qquad \n\in\sphere^2,
\end{equation}
where
\[
  \bs G_0=\operatorname{diag}(1,1.3,1.6),\qquad
  \bs G_1=\begin{pmatrix}
    1.5&0.15&0\\
    0.15&1&0.1\\
    0&0.1&0.8
  \end{pmatrix}.
\]
The initial surface is the ellipsoid
\begin{equation}\label{eq:surface-initial-ellipsoid}
  \frac{x_1^2}{1.4^2}+x_2^2+\frac{x_3^2}{0.8^2}=1.
\end{equation}
We subdivide an icosahedron and project its vertices onto the ellipsoid, giving meshes with $N=42,162,642,2562$ vertices. Let $h_\ell=h_0/2^\ell$, with $h_0=1/2$ and $\ell=0,1,2,3$, denote the subdivision parameter. We take $\tau_\ell=Ch_\ell^2$ for $C=0.08$ and $0.04$, and compare the solutions at $t=0.1,0.2,0.3$.

As in the curve experiments, we measure geometric error by the manifold distance. For surfaces this is the volume of the symmetric difference of the enclosed regions:
\begin{equation}\label{eq:surface-manifold-distance}
  e^{h_\ell}(t)
  \coloneqq M\bigl(\Gamma_{h_\ell,\tau_\ell}(t),
           \Gamma_{h_{\ell+1},\tau_{\ell+1}}(t)\bigr)
  =\bigl|\Omega_{h_\ell,\tau_\ell}(t)\mathbin{\triangle}
         \Omega_{h_{\ell+1},\tau_{\ell+1}}(t)\bigr|,
  \qquad \ell=0,1,2.
\end{equation}
The finest mesh is used to form the last successive-mesh distance. In the tables, $h$ is the subdivision parameter, while the order is calculated from the largest initial edge length $h_{\max,\ell}$:
\begin{equation}\label{eq:surface-convergence-order}
  p_{h_{\max}}
  =\frac{\log\bigl(e^{h_{\ell-1}}(t)/e^{h_\ell}(t)\bigr)}
         {\log\bigl(h_{\max,\ell-1}/h_{\max,\ell}\bigr)},
  \qquad \ell=1,2.
\end{equation}
All rates are computed from unrounded distances. The nonlinear tolerance is $10^{-11}$, and the curvature is evolved as an independent unknown throughout.


\begin{table}[!htbp]
\centering\small
\renewcommand{\arraystretch}{1.35}
\setlength{\tabcolsep}{5pt}
\caption{Error and convergence rates for the $L^2$ flow of surfaces.}\label{tab:surface-coupled-l2}
\begin{tabular}{cccccccc}
\toprule
$C$ & $(h,\tau)$ & $e^h(t=0.1)$ & order & $e^h(t=0.2)$ & order & $e^h(t=0.3)$ & order \\
\midrule
\multirow{3}{*}{0.08} & $(h_0,\tau_0)$ & 2.58e-1 & -- & 2.69e-2 & -- & 1.20e-2 & -- \\
 & $(\frac{h_0}{2},\frac{\tau_0}{4})$ & 7.38e-2 & 1.946 & 6.70e-3 & 2.161 & 3.65e-3 & 1.848 \\
 & $(\frac{h_0}{2^2},\frac{\tau_0}{4^2})$ & 1.90e-2 & 1.994 & 1.59e-3 & 2.119 & 9.70e-4 & 1.951 \\
\midrule
\multirow{3}{*}{0.04} & $(h_0,\tau_0)$ & 3.16e-1 & -- & 8.40e-2 & -- & 1.15e-2 & -- \\
 & $(\frac{h_0}{2},\frac{\tau_0}{4})$ & 8.98e-2 & 1.955 & 2.90e-2 & 1.655 & 3.16e-3 & 2.003 \\
 & $(\frac{h_0}{2^2},\frac{\tau_0}{4^2})$ & 2.31e-2 & 1.997 & 7.80e-3 & 1.931 & 8.21e-4 & 1.982 \\
\bottomrule
\end{tabular}
\end{table}

\begin{table}[!htbp]
\centering\small
\renewcommand{\arraystretch}{1.35}
\setlength{\tabcolsep}{5pt}
\caption{Error and convergence rates for surface diffusion.}\label{tab:surface-coupled-diffusion}
\begin{tabular}{cccccccc}
\toprule
$C$ & $(h,\tau)$ & $e^h(t=0.1)$ & order & $e^h(t=0.2)$ & order & $e^h(t=0.3)$ & order \\
\midrule
\multirow{3}{*}{0.08} & $(h_0,\tau_0)$ & 4.57e-1 & -- & 4.35e-1 & -- & 4.35e-1 & -- \\
 & $(\frac{h_0}{2},\frac{\tau_0}{4})$ & 1.23e-1 & 2.037 & 1.18e-1 & 2.024 & 1.18e-1 & 2.023 \\
 & $(\frac{h_0}{2^2},\frac{\tau_0}{4^2})$ & 3.17e-2 & 1.999 & 3.04e-2 & 2.002 & 3.04e-2 & 2.003 \\
\midrule
\multirow{3}{*}{0.04} & $(h_0,\tau_0)$ & 4.37e-1 & -- & 4.35e-1 & -- & 4.35e-1 & -- \\
 & $(\frac{h_0}{2},\frac{\tau_0}{4})$ & 1.19e-1 & 2.024 & 1.18e-1 & 2.023 & 1.18e-1 & 2.023 \\
 & $(\frac{h_0}{2^2},\frac{\tau_0}{4^2})$ & 3.06e-2 & 1.997 & 3.03e-2 & 2.003 & 3.03e-2 & 2.003 \\
\bottomrule
\end{tabular}
\end{table}

Tables~\ref{tab:surface-coupled-l2} and~\ref{tab:surface-coupled-diffusion}
exhibit second-order convergence in the manifold distance on the finer
meshes. The $L^2$
rates vary more with the observation time and $C$, whereas the fine-mesh
surface-diffusion rates remain between $1.997$ and $2.003$.

\subsection{Volume and energy change}\label{sec:energy-volume}

To verify the structure-preserving properties of the schemes, we introduce the normalized volume loss and the normalized energy as follows: 
\begin{equation}\label{eq:normalized-volume-energy}
  \frac{\Delta V^h(t)}{V^h(0)}\Big|_{t=t_m}\coloneqq\frac{V^m-V^0}{V^0},\qquad
  \frac{W^h(t)}{W^h(0)}\Big|_{t=t_m}\coloneqq\frac{W^m}{W^0}.
\end{equation}

We use the ellipsoid \eqref{eq:surface-initial-ellipsoid} and the coupled
energy \eqref{eq:surface-coupled-density}, with $\tau=0.08h^2$ for
$h=2^{-1},2^{-2},2^{-3}$. The volume histories and normalized energy are
shown in \cref{fig:dVcnt,fig:energy}.

\begin{figure}[htbp]
  \centering
  \includegraphics[width=0.5\textwidth]{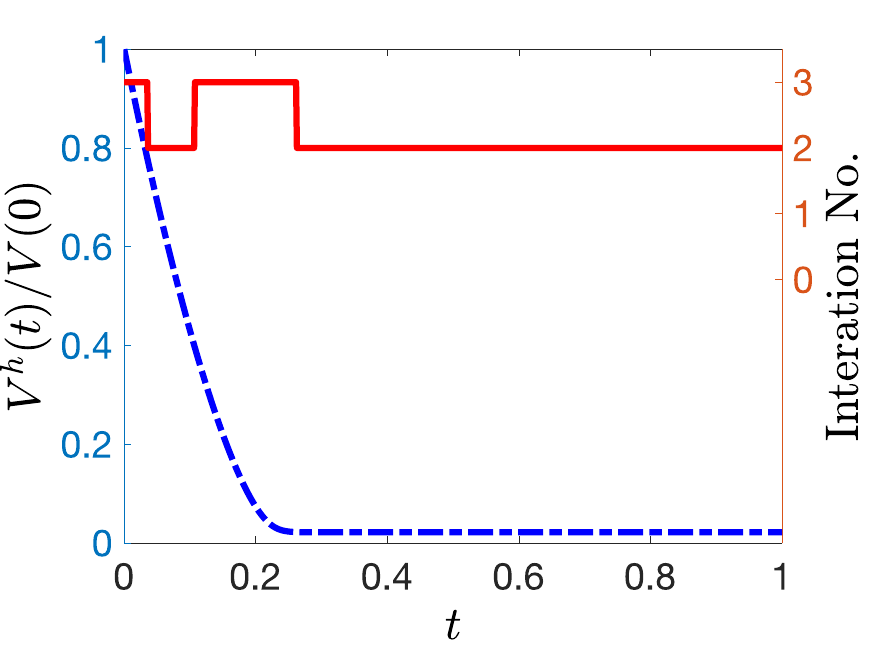}\includegraphics[width=0.5\textwidth]{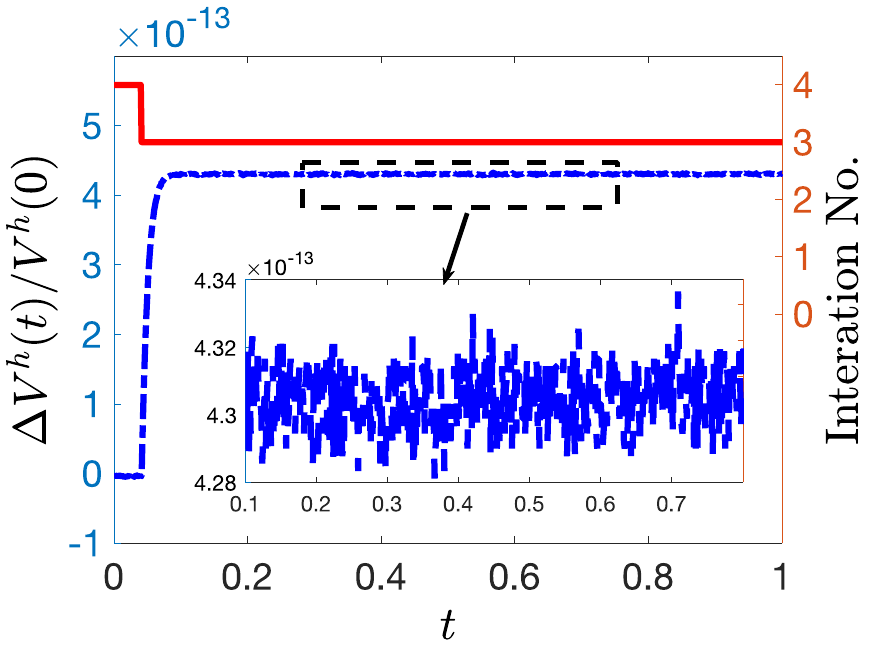}
  \caption{Volume histories and Newton iteration counts with $h=2^{-3}$:
  (left) $V^h(t)/V^h(0)$ for the $L^2$ flow;
  (right) $\Delta V^h(t)/V^h(0)$ for surface diffusion.}
  \label{fig:dVcnt}
\end{figure}

\begin{figure}[htbp]
  \centering
  \includegraphics[width=0.5\textwidth]{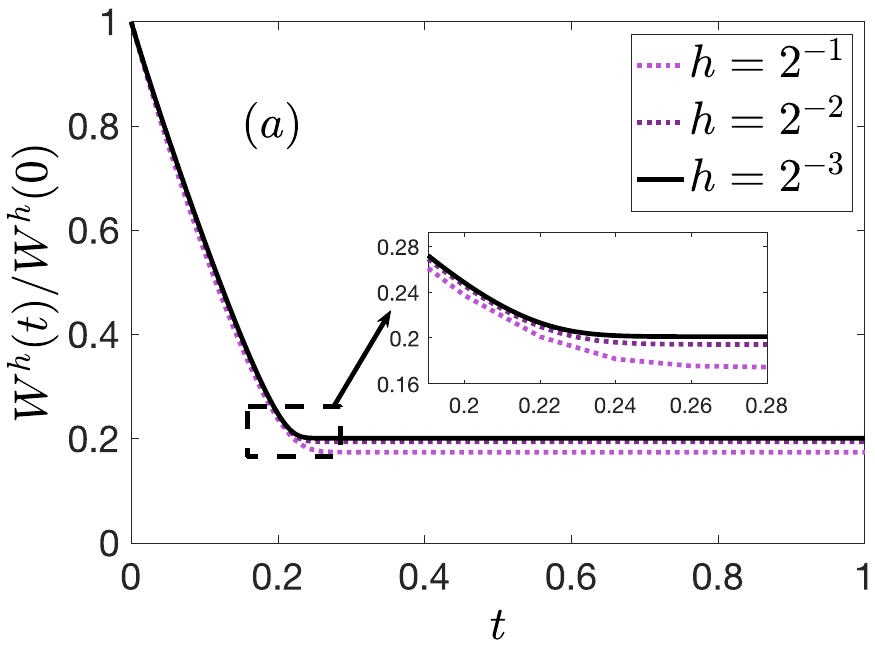}\includegraphics[width=0.5\textwidth]{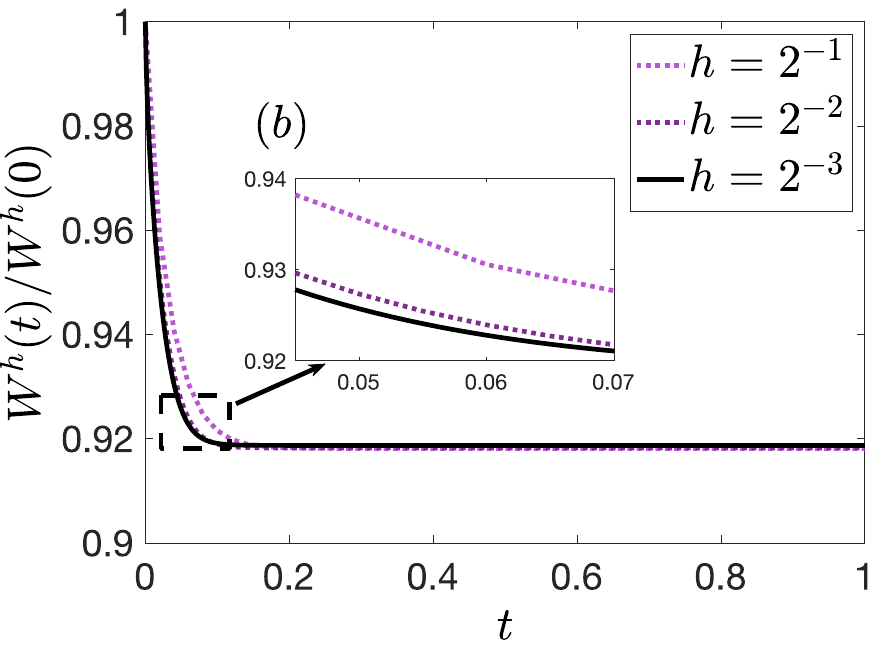}
  \caption{Normalized energy for: (a) $L^2$ flow; (b) surface diffusion.}
  \label{fig:energy}
\end{figure}

It can be observed that the normalized energy decreases monotonically for both the $L^2$ flow and surface diffusion. For surface diffusion, the volume is well preserved, with the volume loss controlled at the order of $10^{-13}$, close to machine precision. The number of Newton iterations is also plotted in \cref{fig:dVcnt}. It decreases to a steady value of about $2$ for the $L^2$ flow and $3$ for surface diffusion, indicating that the nonlinear solver is efficient and robust.

Unlike mean curvature flow, which shrinks a smooth closed convex surface to a point, the $L^2$ flow considered here exhibits a
positive volume plateau. This behavior can be explained by the
competing contributions to the energy \eqref{eq:surface-coupled-density}. Under a uniform rescaling $\Gamma \mapsto \lambda\Gamma$, the surface, quadratic-curvature, and quartic-curvature energy contributions scale as \begin{equation}
  W(\lambda\Gamma) = \lambda^2 W_{\text{surf}}(\Gamma) +  W_{H^2}(\Gamma) + \lambda^{-2} W_{H^4}(\Gamma).
\end{equation} Thus, while the surface energy favors
contraction, the quartic-curvature term increasingly penalizes
further shrinkage as the surface becomes smaller. Their competition
provides a finite preferred scale, consistent with the observed
leveling off of both the energy and the enclosed volume as the
surface approaches an equilibrium configuration.

To assess the computational cost, we use the same ellipsoid and coupled
density with $\tau=0.00125$ and $T=0.02$.
Table~\ref{tab:computational-cost} reports the median wall-clock time per
step over three serial runs on an Apple M1 Pro with 16\,GB RAM, using
Python~3.14 and SciPy~1.18.1 with one BLAS thread.
Timings include assembly, Newton iteration, and step diagnostics;
initialization and file output are excluded. The nonlinear tolerance is
$10^{-11}$. Here $\bar I$ and $I_{\max}$ denote the mean and maximum
Newton iteration counts over the time steps.

\begin{table}[!htbp]
\centering\small
\caption{Wall-clock time per step and Newton iterations.}\label{tab:computational-cost}
\begin{tabular}{crrrrrr}
\toprule
 & \multicolumn{3}{c}{$L^2$ flow} & \multicolumn{3}{c}{Surface diffusion} \\
\cmidrule(lr){2-4}\cmidrule(lr){5-7}
$N$ & Time (s) & $\bar I$ & $I_{\max}$ & Time (s) & $\bar I$ & $I_{\max}$ \\
\midrule
162 & 0.0203 & 2.81 & 3 & 0.0447 & 4.00 & 4 \\
642 & 0.102 & 3.00 & 3 & 0.214 & 4.00 & 4 \\
2562 & 0.859 & 3.00 & 3 & 1.57 & 4.00 & 4 \\
\bottomrule
\end{tabular}
\end{table}

\FloatBarrier

\subsection{Morphology of the evolving surfaces}\label{sec:morphology}

We next examine the shapes produced by the $L^2$ flow and surface
diffusion. The colors indicate the discrete curvature $H^m$.
The first three examples use separable surface and bending
energies. The last example introduces orientation-dependent bending
rigidity, coupling the normal and curvature. The ellipsoid and six-lobed surface are discretized with
$642$ vertices and evolved with $\tau=0.00125$. To monitor the mesh,
we use the triangle quality
\begin{equation}\label{eq:mesh-quality}
  q_\sigma=\frac{4\sqrt{3}|\sigma|}{\ell_1^2+\ell_2^2+\ell_3^2},
  \qquad q_{\min}=\min_{\sigma\subset\Gamma^m}q_\sigma,
\end{equation}
where $\ell_1,\ell_2,\ell_3$ are the edge lengths of $\sigma$.
The value is $1$ for an equilateral triangle and tends to $0$ as the
triangle degenerates. We use $q_{\min}<0.05$ as a stopping criterion.

\begin{samepage}
We first consider the ellipsoid \eqref{eq:surface-initial-ellipsoid}
with the regularized anisotropic energy density
\begin{equation}
  f(\n, H)=\gamma(\n)+\frac{\epsilon^2}{2}H^2,\qquad \gamma(\n)=1+\frac{1}{4}(n_1^3+n_2^3+n_3^3),\qquad \epsilon^2=0.05.
\end{equation}
Figures~\ref{fig:morph-ellipsoid-L2} and~\ref{fig:morph-ellipsoid-SD}
display its evolution under the $L^2$ flow and surface diffusion,
respectively. The ellipsoid shrinks under the $L^2$ flow, whereas under
surface diffusion it evolves towards a nearly stationary anisotropic shape.
\end{samepage}

\begin{figure}[!htbp]
  \centering
  \includegraphics[width=1\textwidth]{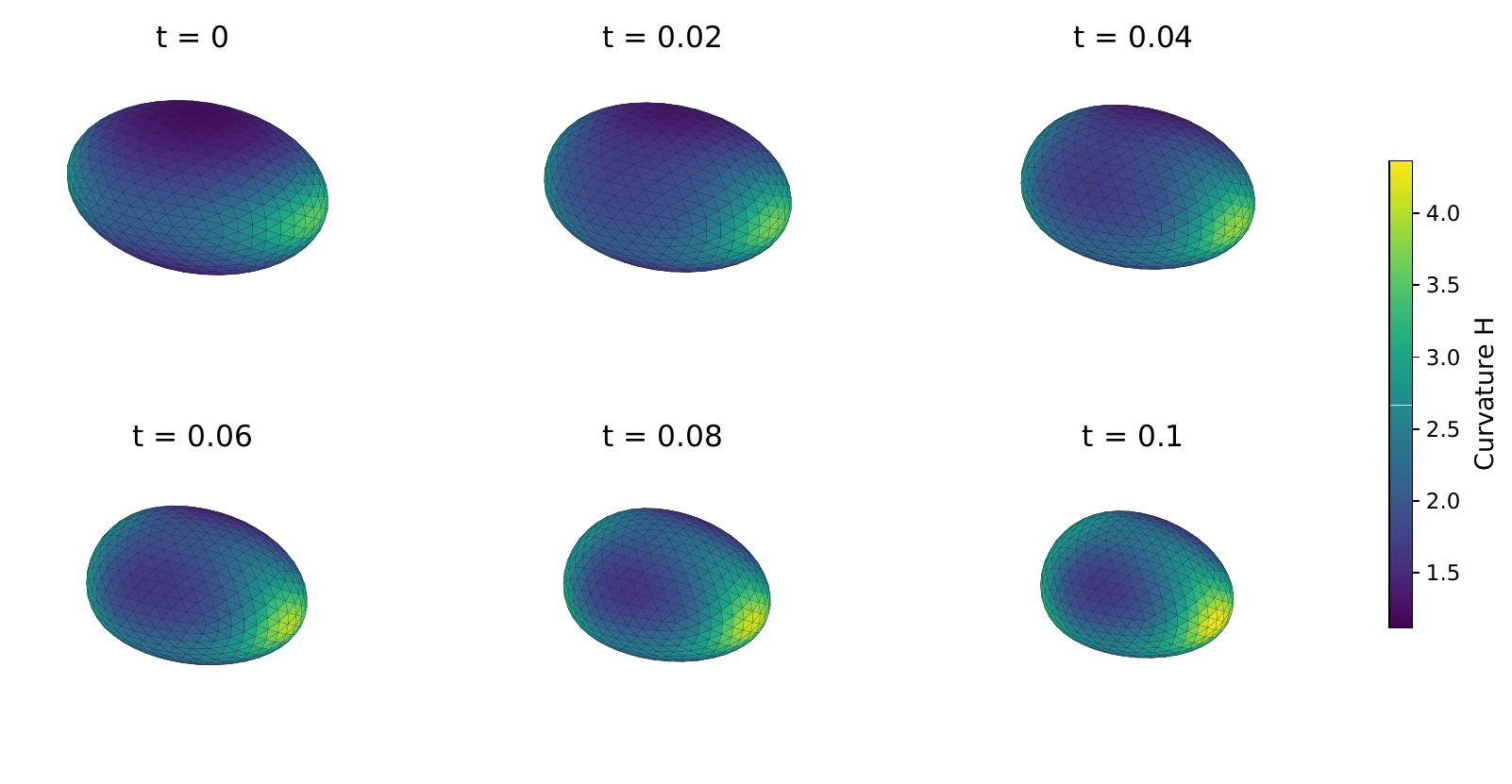}
  \caption{Evolution of the ellipsoid under the $L^2$ flow.}
  \label{fig:morph-ellipsoid-L2}
\end{figure}

\begin{figure}[!htbp]
  \centering
  \includegraphics[width=1\textwidth]{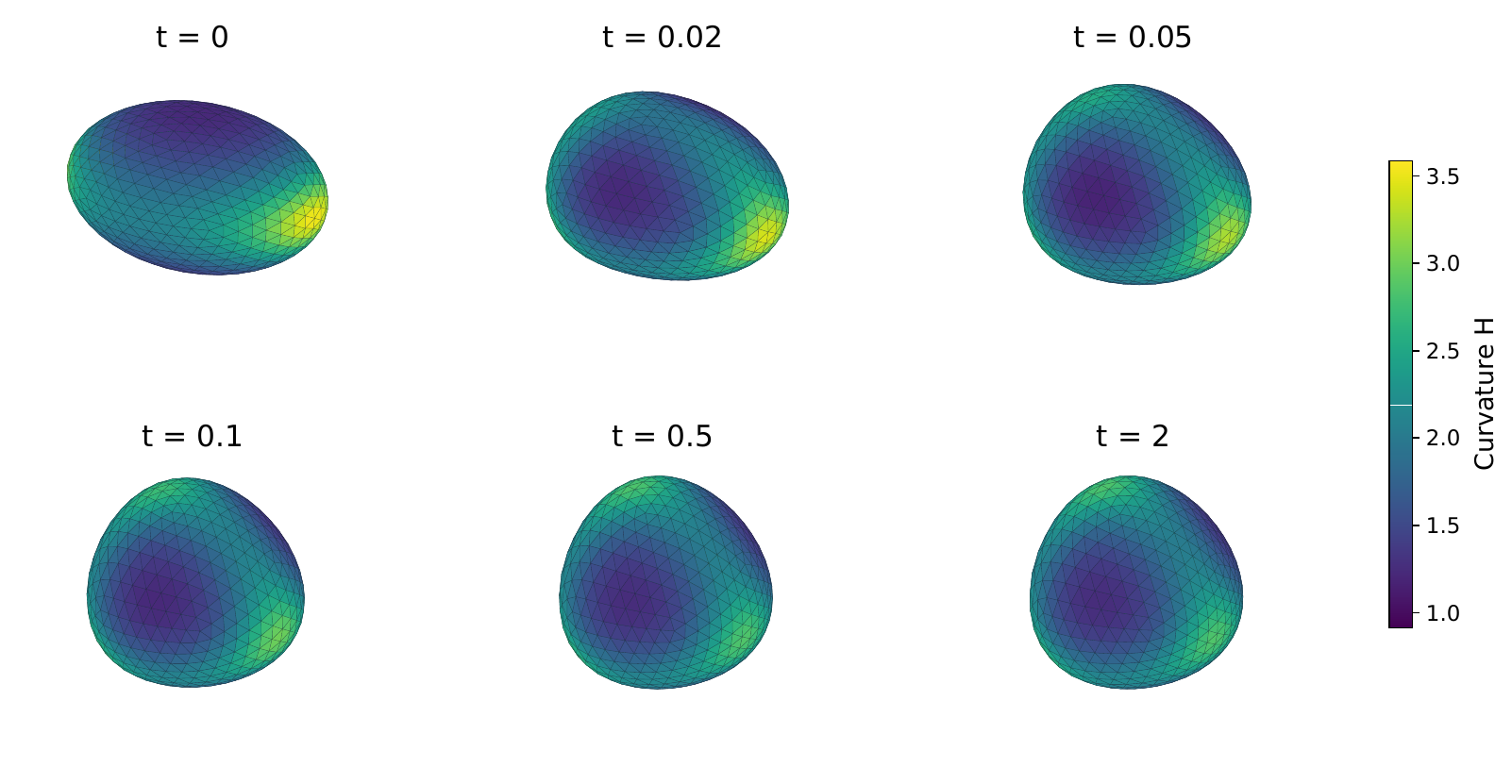}
  \caption{Evolution of the ellipsoid under surface diffusion.}
  \label{fig:morph-ellipsoid-SD}
\end{figure}

\FloatBarrier

\begin{samepage}
Next we consider a six-lobed surface with a regularized anisotropic
energy of cubic symmetry,
\begin{equation}
  f(\n, H)=\gamma(\n)+\frac{\epsilon^2}{2}H^2,\qquad \gamma(\n)=1+\frac{1}{4}(n_1^4+n_2^4+n_3^4),\qquad \epsilon^2=0.05.
\end{equation}
Figures~\ref{fig:morph-6lobed-L2} and~\ref{fig:morph-6lobed-SD} show the evolution of the six-lobed surface under the $L^2$ flow and surface diffusion, respectively. Both flows smooth out the initial lobes and indentations. Similar to the ellipsoid case, the surface continues to shrink under the $L^2$ flow.
Under surface diffusion, it preserves its enclosed volume while relaxing
towards an anisotropic equilibrium.
\end{samepage}

\begin{figure}[!htbp]
  \centering
  \includegraphics[width=1\textwidth]{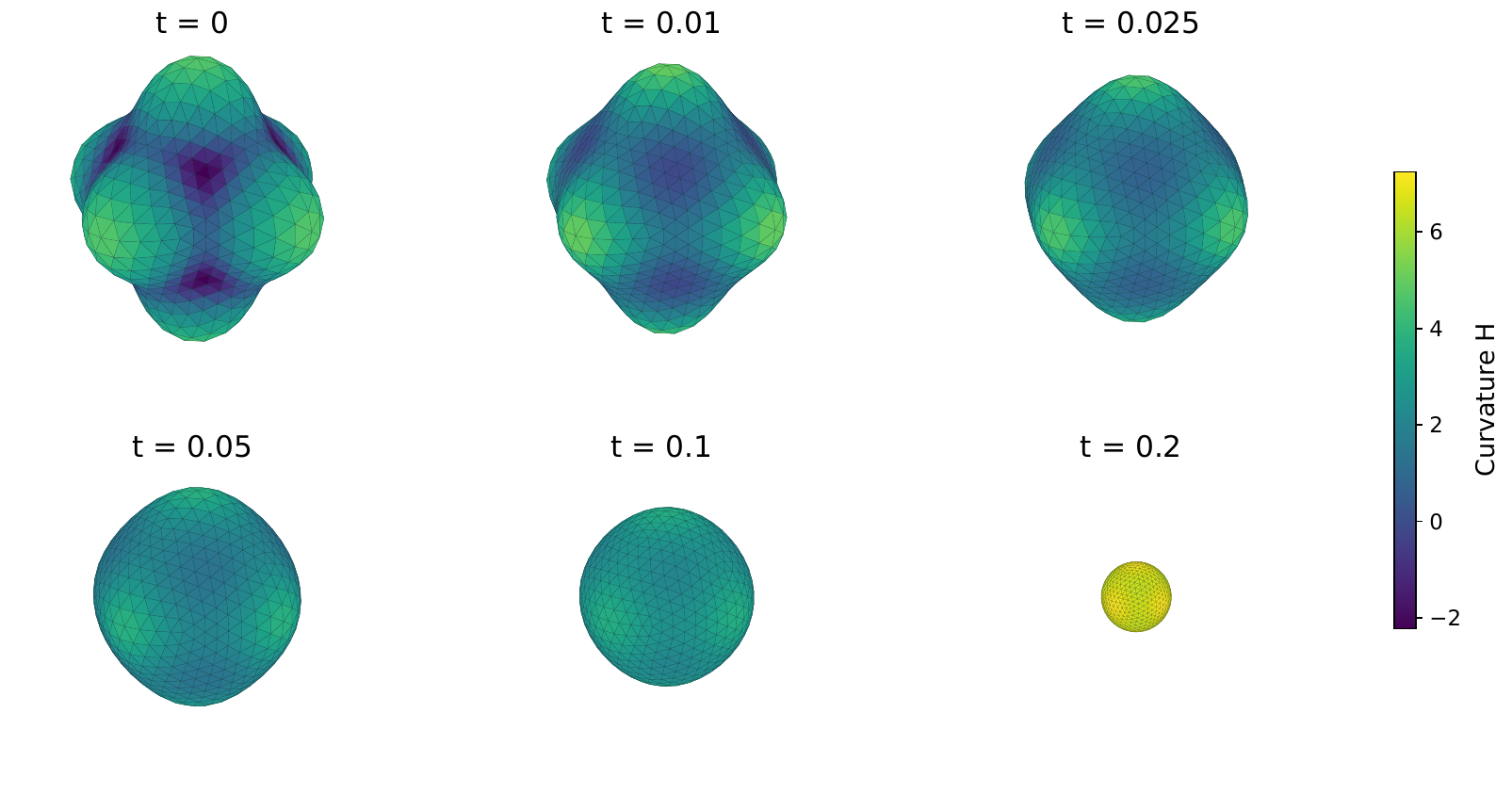}
  \caption{Evolution of the six-lobed surface under the $L^2$ flow.}
  \label{fig:morph-6lobed-L2}
\end{figure}

\begin{figure}[!htbp]
  \centering
  \includegraphics[width=1\textwidth]{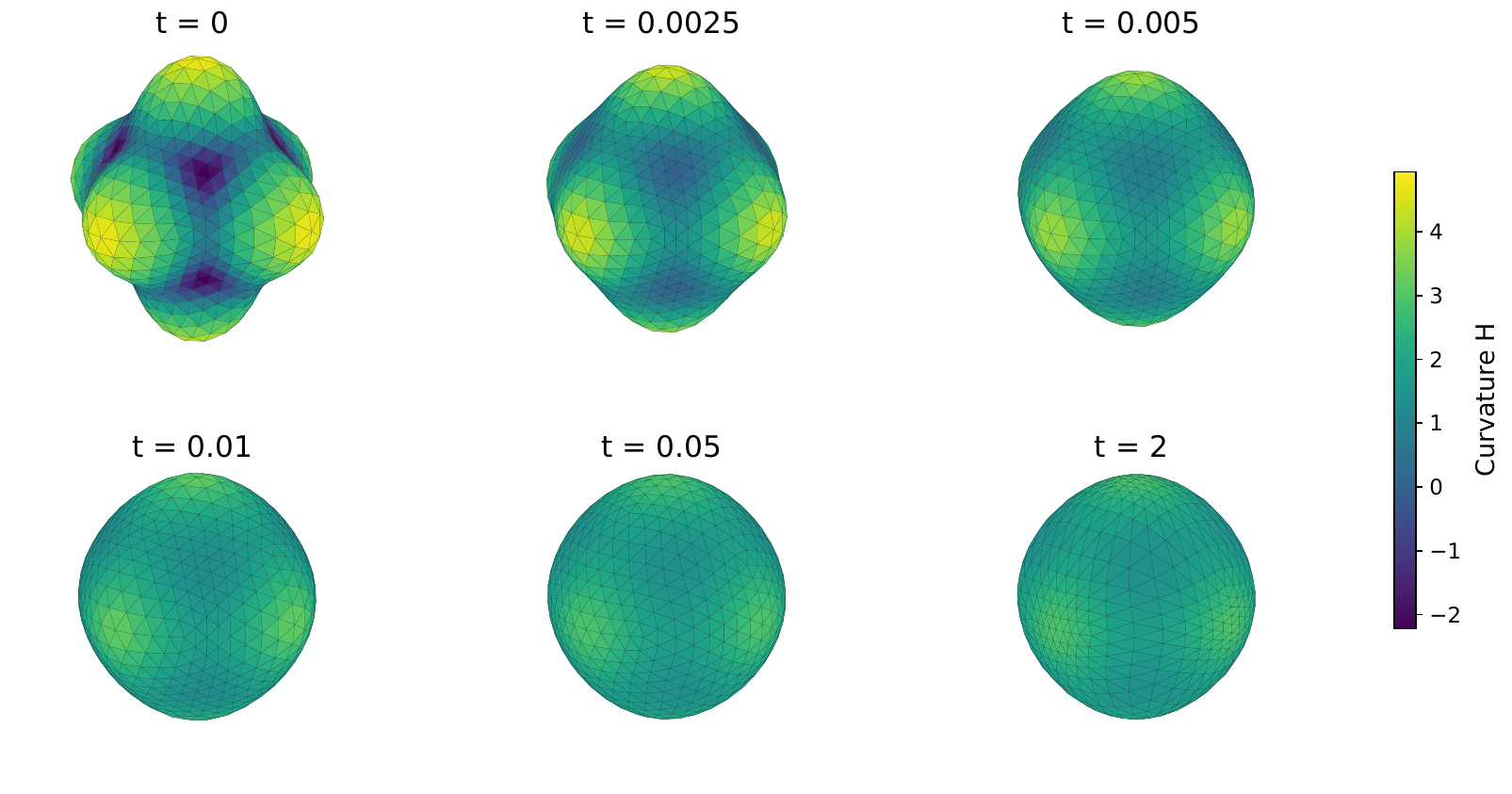}
  \caption{Evolution of the six-lobed surface under surface diffusion.}
  \label{fig:morph-6lobed-SD}
\end{figure}

\FloatBarrier

\begin{samepage}
To examine the evolution of a surface with nontrivial topology, we take an initial torus with major radius $R=1$ and minor radius $r=0.6$. The energy density is chosen as
\begin{equation}
  f(\n, H)=\gamma(\n)+\frac{\epsilon^2}{2}H^2,\qquad \gamma(\n)=1+\frac{1}{8}(n_1^3+n_2^3+n_3^3),\qquad \epsilon^2=0.05.
\end{equation}
As shown in Figure~\ref{fig:morph-torus-L2}, the torus contracts
and its central hole progressively narrows. Both the enclosed volume
and the energy decrease monotonically during the evolution.
\end{samepage}

\begin{figure}[!htbp]
  \centering
  \includegraphics[width=1\textwidth]{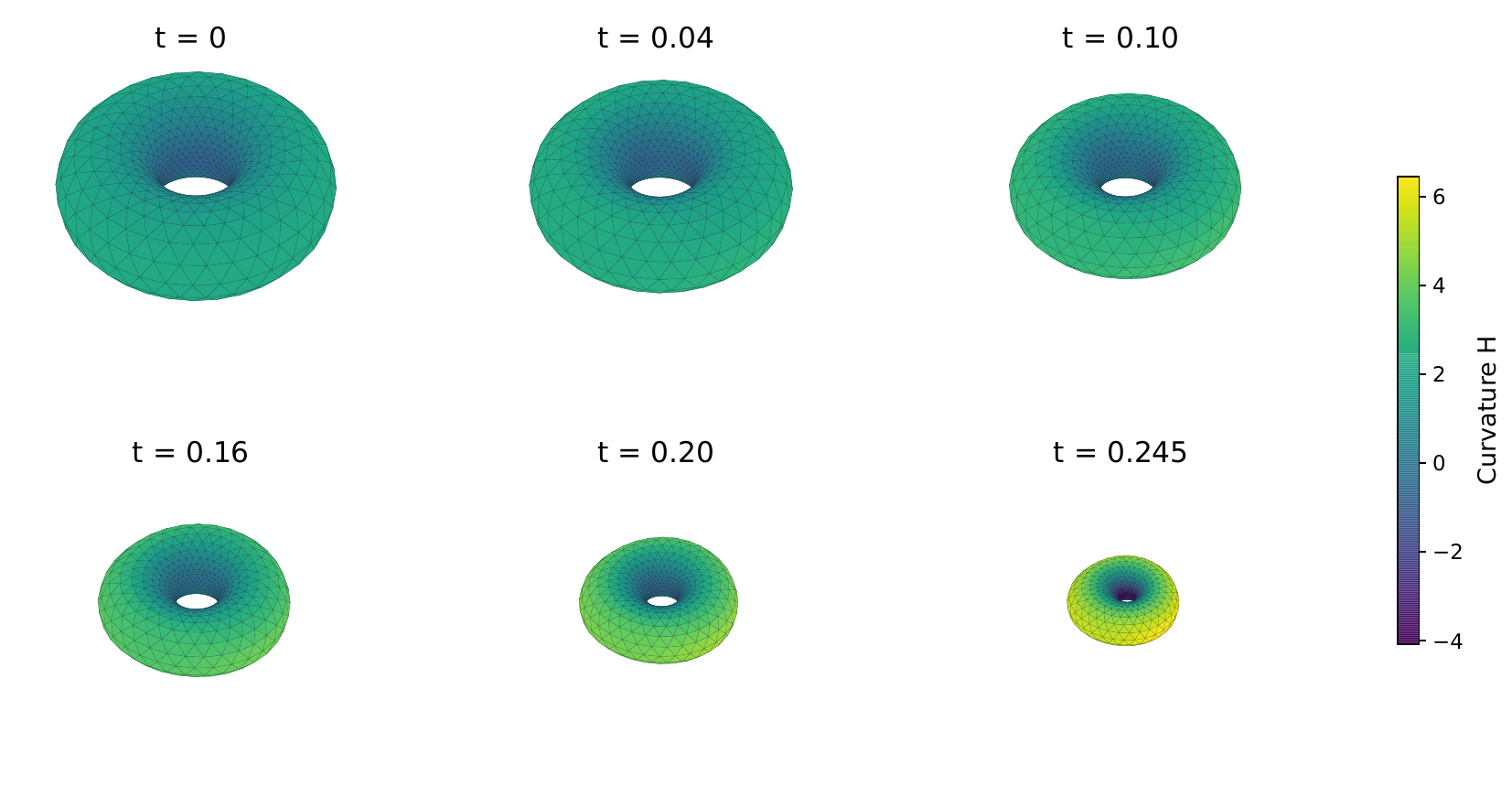}\\
  \includegraphics[width=1\textwidth]{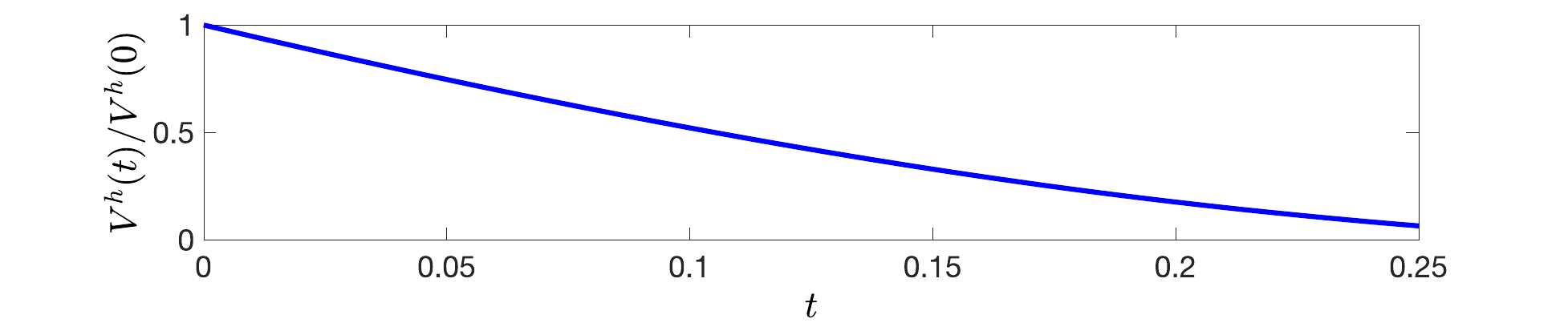}\\
  \includegraphics[width=1\textwidth]{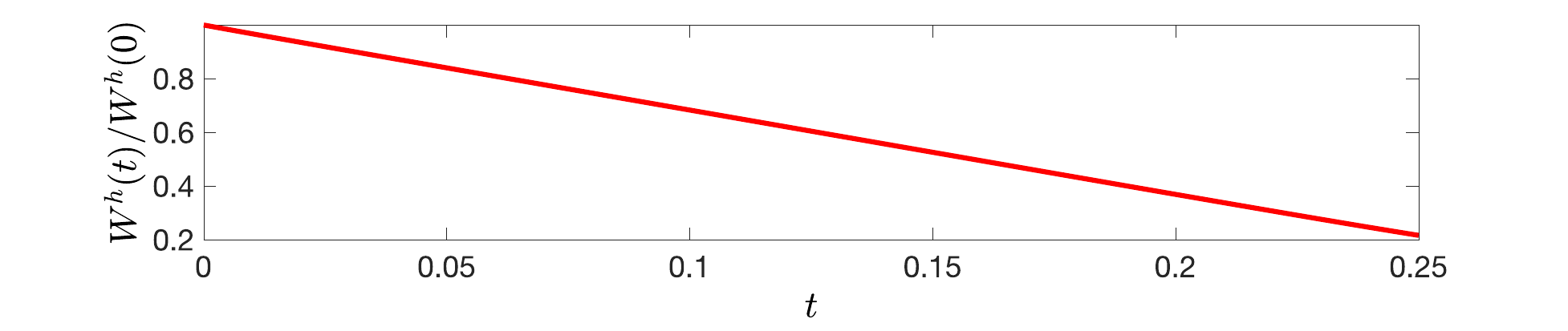}
  \caption{Evolution of the torus under the $L^2$ flow with $896$ vertices and $\tau=0.005$.}
  \label{fig:morph-torus-L2}
\end{figure}

\clearpage
\begin{samepage}
Finally, we consider surface diffusion of a biconcave disk with an
orientation-dependent Helfrich-type energy. The classical Helfrich model
for lipid bilayers penalizes deviations from a spontaneous curvature
\cite{helfrich1973elastic}, which may reflect leaflet asymmetry
\cite{seguin2014microphysical}. The biconcave initial shape is motivated by
the connection between bending energy and red-cell shapes
\cite{canham1970minimum}. We choose
\begin{equation}\label{eq:biconcave-helfrich-density}
  \begin{aligned}
    f(\n,H)&=\gamma(\n)+\frac{\epsilon^2}{2}a(\n)(H-H_\ast)^2,\\
    \gamma(\n)&=1+\frac14\sum_{i=1}^3 n_i^4,
    \qquad a(\n)=1+2n_3^2,
  \end{aligned}
\end{equation}
with $\epsilon^2=0.05$ and $H_\ast=2$.

The density is nonseparable and strictly convex in $H$.
For a given curvature deviation, the bending penalty is three times larger
when $\n$ is parallel to $\boldsymbol e_3$ than when it is perpendicular
to $\boldsymbol e_3$. The orientation dependence and coefficients are chosen for this
experiment and are not fitted to a particular membrane material. No membrane-area constraint
is imposed.

We use $2562$ vertices and $\tau=0.005$.
Figure~\ref{fig:morph-biconcave-SD} shows that the central depressions
gradually disappear as the disk becomes more rounded. The energy decreases
rapidly at first and then more slowly, reaching
$W^h(0.4)/W^h(0)\approx0.6748$, while the relative volume change remains
below $1.5\times10^{-13}$ in magnitude. The computation is stopped at
$t=0.4$ when the mesh-quality threshold is reached.
\end{samepage}

\begingroup
\setlength{\intextsep}{8pt}
\begin{figure}[!htbp]
  \centering
  \includegraphics[width=0.90\textwidth,trim=0 35bp 0 10bp,clip]{figure/morph/morph3SD.pdf}\par\nointerlineskip
  \includegraphics[width=0.90\textwidth,trim=0 5bp 0 2bp,clip]{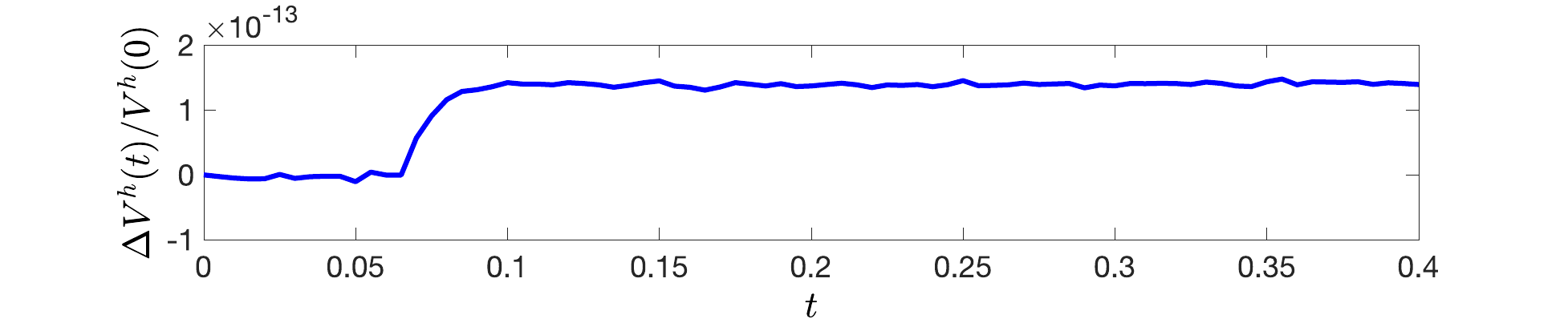}\par\nointerlineskip
  \includegraphics[width=0.90\textwidth,trim=0 5bp 0 5bp,clip]{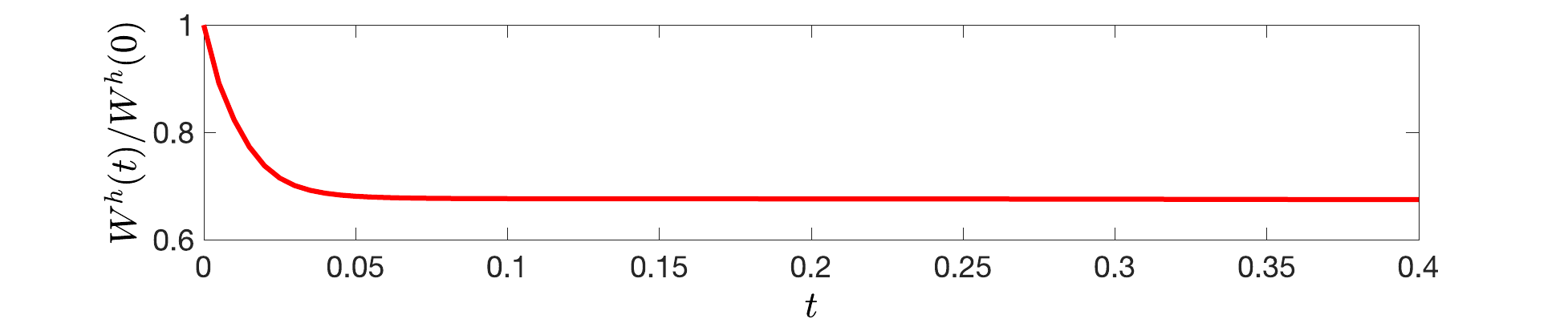}
  \caption{Evolution of the biconcave disk under surface diffusion with energy \eqref{eq:biconcave-helfrich-density}.}
  \label{fig:morph-biconcave-SD}
\end{figure}
\endgroup

\FloatBarrier

\clearpage
\section{Concluding remarks}\label{sec:conclusions}

We have developed a unified structure-preserving PFEM framework for
geometric flows with coupled orientation and curvature dependence.
Closed curves in two dimensions and closed surfaces in three dimensions
are treated by the same piecewise linear formulation. The framework
accommodates the $L^2$ flow, curve or surface diffusion, and the
area- or volume-constrained $L^2$ flow. Under the stated directional
condition and convexity in curvature, the fully discrete schemes
dissipate the energy without a time-step restriction. The diffusion and
constrained schemes also preserve the enclosed area or volume exactly.
These properties extend to nonseparable densities and non-even
anisotropies.

The numerical results support the predicted dissipation and conservation
laws and exhibit approximately second-order convergence in the manifold
distance.
The morphological experiments illustrate how a common
discretization can describe shape relaxation driven by anisotropic
surface tension and bending, including an orientation-dependent
Helfrich-type energy. In particular, the method permits the energy
density to be varied without changing the underlying finite element
spaces or the discrete curvature evolution equation.

The energy estimate does not by itself control mesh quality or ensure
solvability of the nonlinear system. Mesh deterioration in some of the
longer computations therefore remains a practical limitation. Developing
mesh adaptation compatible with the discrete energy and conservation
laws, together with a convergence analysis for the coupled scheme, would
be useful steps toward reliable computations over longer time intervals.

\bibliographystyle{siamplain}
\bibliography{Main}
\end{document}